\documentclass[a4paper,10pt]{article}
\usepackage{amsbsy,amscd,graphicx}
\usepackage{hyperref}
\usepackage{xcolor}
\usepackage{booktabs}
\usepackage{graphics,epsfig}
\usepackage[shortlabels]{enumitem}
\usepackage{bbm}
\usepackage{bm} 

\usepackage{amssymb, bbm, booktabs, subcaption}
\usepackage[english]{babel}
\usepackage{amsmath,amsfonts, amsthm}
\usepackage[shortlabels]{enumitem}

\theoremstyle{plain}
\newtheorem{theo}{Theorem}[section]                                      
\newtheorem{prop}[theo]{Proposition}                          
\newtheorem{lem}[theo]{Lemma}
\newtheorem{cor}[theo]{Corollary}

\theoremstyle{definition}

\newtheorem{hyp}[theo]{Assumption}

\theoremstyle{remark}
\newtheorem{rem}[theo]{Remark}

\makeatletter \@addtoreset{equation}{section} \makeatother

\begin{document}
\title{Individual based SIS models \linebreak on (not so) dense large random networks
}
\author{Jean-Fran\c{c}ois Delmas \thanks{Cermics, Ecole des Ponts, France; E-mail:
    \texttt{jean-francois.delmas@enpc.fr}}
  \and Paolo Frasca \thanks{Univ.\ Grenoble Alpes, CNRS, Inria, Grenoble
    INP, GIPSA-Lab, 38000 Grenoble, France; E-mail:
    \texttt{paolo.frasca@gipsa-lab.fr}}
  \and Federica Garin \thanks{Univ.\ Grenoble Alpes, Inria, CNRS,
    Grenoble INP, GIPSA-Lab, 38000 Grenoble, France; E-mail:
    \texttt{federica.garin@inria.fr}}
  \and Viet Chi Tran \thanks{LAMA, Univ Gustave Eiffel, Univ Paris Est
    Creteil, CNRS, F-77454 Marne-la-Vall\'ee, France; IRL 3457,
    CRM-CNRS, Universit\'e de Montr\'eal, Canada; E-mail:
    \texttt{chi.tran@univ-eiffel.fr}}
  \and Aur\'{e}lien Velleret \thanks{Université Paris-Saclay, INRAE, MaIAGE, F-78350 Jouy-en-Josas, France; E-mail:
    \texttt{velleret@phare.normalesup.org}, corresponding author} 
  \and Pierre-Andr\'{e} Zitt \thanks{LAMA, Univ Gustave Eiffel, Univ
    Paris Est Creteil, CNRS, F-77454 Marne-la-Vall\'ee, France; E-mail:
    \texttt{Pierre-Andre.Zitt@univ-eiffel.fr}} 
}
\date{\today}

\maketitle	

\begin{abstract} Starting from a stochastic individual-based description of an SIS epidemic spreading on a random network, we study the dynamics when the size $n$ of the network tends to infinity. We recover in the limit an infinite-dimensional integro-differential equation studied by Delmas, Dronnier and Zitt (2022) for an SIS epidemic propagating on a graphon.
	Our work covers the case of dense and sparse graphs, 
	provided that the number of edges grows faster than $n$, but not the case of very sparse graphs with $O(n)$ edges.
	In order to establish our limit theorem, we have to deal with both the convergence of the random graphs to the graphon and the convergence of the stochastic process spreading on top of these random structures: in particular, we propose a coupling between the process of interest and an epidemic that spreads on the complete graph but with a modified infection rate.
\end{abstract}

\noindent Acknowledgments: This work was financed by the Labex B\'ezout (ANR-10-LABX-58) and the COCOON grant (ANR-22-CE48-0011), and by
the platform MODCOV19 of the National Institute of Mathematical Sciences and their Interactions of CNRS. V.C.T. is partly financed by the Chaire ``Mod\'elisation Math\'ematique et Biodiversit\'e'' of Veolia-Ecole Polytechnique-Museum National d'Histoire Naturelle-Fondation X. The research leading to this article was largely performed while A. Velleret was a researcher at LAMA and at GIPSA-Lab.\\
Accepted at ALEA: \url{https://doi.org/10.30757/ALEA.v21-52}

\section{Introduction}

We consider the spread of  diseases with potential reinfections in large
heterogeneous  populations structured  by random  networks. 
We  focus  on  SIS models,  in  which  the population  is
partitioned into  two classes,  namely $S$ for  susceptible individuals
and  $I$  for infected  ones.   Infected  individuals can  transmit  the
disease to  susceptible ones if they  are in contact through  the social
network.  Infected  individuals remain infectious until  their recovery,
which is assumed to be spontaneous. The susceptibility and infectivity of
individuals as well as their degrees in the social network
can be  heterogeneous, which  makes the dynamics  of such  systems
very complex.   In this paper,  we are interested in  establishing limit
theorems showing that, in large population, the possibly complex dynamics
of  the  epidemics can  be  approximated  by  a system  of  integro-differential
equations. 

To be more precise, we start from a stochastic individual-based model of
a  finite population  of  size $n$,  where each  individual  $i$ in  the
population is  characterized by  a feature $x_i$  that belongs  to a generic space
$\mathbb{X}$  (assumed to  be metric, separable and complete).  One  may think of
the feature being a  group label (so that $\mathbb{X}$ can  be discrete), or the individual's
age (so that  $\mathbb{X}=\mathbb{R}_+$), or their location (so that  $\mathbb{X}$ can be the sphere or $\mathbb{R}^d$),
or a combination of those. 
The distribution of features in the population  can be encoded
by the (possibly random) measure:
\[
\mu^{(n)}(dx)=\frac{1}{n} \sum_{i =1}^{n}   \delta_{x_i}(dx).
\]
The  social contacts in the population  are modeled by a  random network, $G^{(n)}$,
which  is fixed  in time:  
a contact between
individuals $i$  and $j$
means that the edge $(i, j)$ belongs to $G^{(n)}$,
which happens with probability $w^{(n)}_E(x_i,  x_j)$ depending  on their respective features
independently of all other edges. 
Each individual $i$ is, at time $t$, in a state, say $E^i_t$, which
is either $S$ (susceptible) or $I$ (infected). 
The initial condition is completely characterized
by the (possibly random) sequences $\mathcal{X}^{(n)}=(x_i)_{i\in [\![1,
	n]\!]}$ and $\mathcal{E}^{(n)}=(E^i_0)_{i\in [\![1,
	n]\!]}$. At time $t$, if $i$ is
infected, that is $E^i_t=I$, then it recovers at rate
$\gamma^{(n)}(x_i)\geq 0$, also depending on its feature; if 
$i$ is susceptible,  that is $E^i_t=S$, then it can be infected by
individual $j$ at rate $w^{(n)}_I(x_i, x_j)$ provided that $i$ and $j$ are
connected (in $G^{(n)}$) and $j$ is infected. At time $t$, we describe the infected
population with the measure $\eta^{I,(n)}_t$ on $\mathbb{X}$ given by:
\begin{equation}
	\label{eq:intro-etaN}
	\eta_t^{I,(n)}(dx) = \frac{1}{n} \sum_{i =1}^{n} \mathbbm{1}_{\{E^i_t=I\}}\, \delta_{x_i}(dx).
\end{equation}
Our main result, see Theorem~\ref{theo_G1} for a precise statement, is
on the convergence of the processes $\eta^{I,(n)}=(\eta^{I,(n)}_t)_{t\in \mathbb{R}_+}$ to
a deterministic process $\eta^I=(\eta^I_t)_{t\in \mathbb{R}_+}$ of the form:
\begin{equation}
	\label{eq:intro-eta}
	\eta^I_t(dx) =u(t,x) \, \mu(dx) , 
\end{equation}
where the probability measure  $\mu(dx)$,
namely the limit of the sequence $\mu^{(n)}$,
represents the probability for an individual of the population taken at
random to have feature~$x$, and $u(t, x)$ represents the probability for  an individual with feature~$x$ to be infected at time $t$. We will show that the function
$(u(t,x), t\in \mathbb{R}_+, x\in \mathbb{X})$ is the unique solution of the 
integro-differential equation:
\begin{equation}
	\label{eq:intro-uT}
	\partial_t u(t, x)= (1-u(t, x))\int_{\mathbb{X}}  u(t, y)\, 
	w(x, y)\, \mu(dy)- \gamma(x) u(t, x),
\end{equation}
with  initial  condition  $u(0,   \cdot)=u_0$,  where  $u_0(x)$  is  the
probability for  an individual  with feature  $x$ to  be infected  at time
$0$. In \eqref{eq:intro-uT}, the function $w(x,y)$  will be interpreted as a transmission kernel from
(infected) individuals with feature $y$ to (susceptible) individuals with feature $x$, and contains information on both the network and the transmission rate.
The function $\gamma(x)$ can be interpreted as the recovery  rate of (infected) individual  with feature  $x$,
as $\gamma^{(n)}$ in the finite population model. For
this convergence  to hold,  we assume, see  Assumptions~\ref{hyp:A1}
and~\ref{hyp:techn}:
\begin{enumerate}[a)]
	\item \textbf{Convergence  of the  initial condition.}  The distribution
	of        the         features   $\mu^{(n)}$     in         the        population
	converges to a limit measure $\mu$
	(for the weak topology on non-negative  measures). The
	initial distribution of the infected population $\eta^{I,(n)}_0$
	converges to a deterministic limit $\eta^I_0$, 
	with  $\eta^I_0(dx)=u_0(x)  \, \mu(dx)$.
	\item \textbf{Uniform convergence and regularity of the limit parameters.} 
	The recovery rate $\gamma^{(n)}$ as well as  
	the average transmission
	rate $w^{(n)}$, given by the following formula for any $(x, y)\in \mathbb{X}^2$:
	\[
	w^{(n)}(x,y)=n\, w^{(n)}_E(x,y) \, w_I^{(n)}(x,y),
	\]
	converge uniformly to $\gamma$ and $w$, which are respectively
	$\mu$-a.e.\ and $\mu^{\otimes 2}$-a.e.\ continuous. The scaling factor $n$ appears naturally as $n\, w^{(n)}_E(x,y)$
	corresponds to the density of contacts 
	between an individual with feature $x$ and the population with feature $y$ (scaled by the measure $\mu(dy)$),
	where each of these contacts corresponds to a potential infection rate of $w^{(n)}_I(x,y)$.
	
	\item \textbf{Control on the infection rate.} 
	We assume the following
	convergence in mean (see Assumption~\ref{hyp:techn} for a precise
	and slightly more general  statement):
	\begin{equation}
		\label{eq:intro-wI=0}
		\lim_{n\rightarrow \infty } \frac{1}{n^2} \sum_{i, j\in [\![1, n]\!]} w^{(n)}_I
		\big(x_i, x_j\big) =0,
	\end{equation}
	meaning that the infection rate per edge is small on average over the population.
\end{enumerate}
\medskip

Our results  cover in particular the  case of dense and  sparse graphs,
when the  total number  of edges is  scaled to be  of order  $n^a$ with
$a=2$ and $a\in (1,2)$  respectively (see Remark~\ref{rem_epsN} below);
this  corresponds  for  example  to  $w^{(n)}_E\equiv  n^{a-2}$  and  thus
$w^{(n)}_I$ of   order  $n^{1-a}$, by Assumption b) above.   Notice   that~\eqref{eq:intro-wI=0}
trivially holds in this case.  However,  our method fails to cover very
sparse graphs corresponding  to $a=1$ for which the number  of edges is
of order $n$ thus $w^{(n)}_I$ of order $1$, as~\eqref{eq:intro-wI=0} does
not hold.   Simulations, see Section~\ref{sec_Vsparse}, 
hint that,
although a deterministic limit may be derived,
this limit is nonetheless not suitably represented  by  the
deterministic process~\eqref{eq:intro-eta}-\eqref{eq:intro-uT}. 
Yet, the very-sparse case triggered a large literature; in this direction, see the reference in the related works section below.

\medskip

In the dense  heterogeneous case ($a=2$), if $w^{(n)}_E$  does not depends
on $n$,  and thus is equal  to $\bar{w}_E$, and if  $w^{(n)}_I=n^{-1} \bar{w}_I$,
for some fixed kernels $\bar{w}_E$ and $\bar{w}_I$,
then  we have $w= \bar{w}_E \bar w_I$. In this case the  sequence of graphs, $G^{(n)}$, which
models the connection in the  population of size $n$, converges towards
a  graphon with  parameter $\bar{w} _E$,  see \cite[Part  3]{lovaszbook}.  The
quantity $\bar{w} _E(x,y)$ is understood as the density of connections between
the  population  with  feature  $x$ and  the  population  with  feature
$y$. Then  the kernel $w$  can be seen as  the result of  weighting the
graphon $\bar{w}_E$ by another kernel $\bar{w}_I$, which represents the interaction
rates between features. See Remark~\ref{rem:bwq} below for further comment
in this direction.

\paragraph{Related  work.}  There  has  been  a  growing  literature  on
epidemics spreading  on graphs, and  the interested reader can  refer to
\cite{andersson,anderssonbritton,ballbrittonlaredopardouxsirltran,house,kissmillersimon,newman,newman_SIAM}. Scaling limits for such processes have
been considered mostly when the characteristics of  the graph are
fixed:   for  instance,   diseases  spreading   on  fixed   lattices  or
configuration model  graphs with fixed degree  distributions. Many points
of     view     have     been    taken:     with     moment     closures
\cite{durrett,georgioukisssimon},   or   using    limit   theorems   for
semi-martingale                                                processes
\cite{ballneal,volz,miller,decreusefonddhersinmoyaltran,jansonluczakwindridge}
or for  branching processes  \cite{barbourreinert}. In  this literature,
the graphs are \textit{very sparse} with a degree distribution that does
not depend on the size $n$ of the population.

In  the meantime,  there has  been also  a growing  interest in  scaling
limits                 for                 random                 graphs
\cite{borgschayeslovaszsosvesztergombi,lovaszbook}.    In  the   present
paper,  we  model the  social network  by a  graph and  obtain
for the dense  graph case (for $a=2$ above), as limit when $n$  tends to infinity, a
dynamical  system propagating  on a  graphon. Limit
theorems for  dynamical systems  on graphs  converging to  graphons have
been  considered in  \cite{BLRT22,KHT22,kuehnthrom}. In \cite{kuehnthrom,BLRT22}, this is done for  systems of  ordinary
differential  equations on  the graph  (or  when the  randomness of  the
individual-based model  has been averaged  first, and the graph  limit is
considered    in     a    second     time). In \cite{KHT22}, starting from the stochastic individual-based dynamics  on a  random  graph as in the present work, it is proved that after a proper rescaling so that the graph converges to a graphon, the evolution  of the  susceptible and  infected populations converge to the 
integro-differential equation~\eqref{eq:intro-uT},
first  introduced   and
studied in  \cite{DDZ22a}-\cite{DDZ22d}. We  improve the  results from
\cite{KHT22} by taking  into account propagation rates depending on the
features, considering  weaker assumptions  on the  parameters $\gamma$
and  $w$ of  the model,  and most  importantly identifying  the global
condition~\eqref{eq:intro-wI=0}   on  the   mean infection   rate  as
sufficient  to get  this convergence,  generalizing in  particular the
possible   range   of   the    scaling   parameters.   We   refer   to
Remark~\ref{rem:comment}     for     further     comments.     

Eventually, our assumptions  on the regularity of $w$ allow  for a large
variety of random  graphs and of feature dependencies  that include, for
instance,  geometric random  graphs \cite{penrose_randomgeometricgraphs}
or stochastic block models \cite{abbe}, see Remark~\ref{rem_Vgraphs} below.

We    also     cite
\cite{alettinaldi,naldipatane,vizuetefrascagarin}    for   deterministic
epidemic        dynamics        on        graphons,        and        
\cite{aurellcarmonadayaniklilauriere} for  a finite  agent model  on the
graphon.  See also  \cite{vanderhofstadjanssenleeuwaarden} for  epidemic
propagation on an Erd\H{o}s-R\'enyi random graph in the critical window.

Scaling limits  of the SIS epidemic  process for a
large  number   of  individuals  with  homogeneous   interactions  are
well-described, including  the study  of the dynamical  system derived
from the law of large numbers, the deviations prescribed by the central
limit theorem or by the large deviation theory (see e.g. \cite[Chapter 8.2]{anderssonbritton} or \cite[Part I]{ballbrittonlaredopardouxsirltran}). Taking into account features with finite possible values is a direct extension. To  tackle a more  general feature dependency,  we follow
similar lines as in  \cite{fourniermeleard}, where a mean-field kernel
is  introduced  to  represent  local effects  of  competition  between
plants.

\paragraph{Outline.}   In Section~\ref{sec:def-resultats},  we present
the stochastic individual-based  model of SIS epidemic  spreading on a
random network. Two objects, both of random nature, are at the core of
the study: a random  graph $G^{(n)}$ of size $n\in \mathbb{N}$  and a sequence of
stochastic measure-valued processes $(\eta^{(n)})_{n\in \mathbb{N}}$.  We discuss
the sparsity of the random graph and state our main convergence result
in Theorem \ref{theo_G1}.  The proof  of the theorem, which is carried
out     in    Section~\ref{sec:proofs},     follows    a     classical
tightness-uniqueness   scheme.    First,   tightness  is   proved   in
Section~\ref{sec_tight},    which    implies   that    the    sequence
$(\eta^{(n)})_{n\in    \mathbb{N}}$   is    relatively    compact.    Next,    in
Section~\ref{sec_cvg}, we show that  the limiting values are solutions
of a  deterministic integro-differential  equation: proving  this fact
requires a careful coupling argument that is developed in Appendix \ref{sec:appendix}.    Finally,  the  uniqueness   of  the solution to the
limiting       equation,      which       is      established       in
Section~\ref{sec:uniqueness},  allows  us   to  conclude.  A  detailed
discussion on the scope of our results and numerical simulations, with
an  emphasis  on  the  limit  behavior for  very  sparse  graphs,  are
presented in Section~\ref{sec:discussion}.

\section{Definition of the model and of the limiting
	equation}\label{sec:def-resultats}

In this section we formally define the relevant mathematical objects and state our main result. After defining some useful notation in Section~\ref{sect:notation}, we define the stochastic individual-based model in Section~\ref{sec:IBM} and finally in Section~\ref{sec:cv_largepop-graphon} we define the limit graphon-based integro-differential equation and we state the main convergence result. 

\subsection{Notation}\label{sect:notation}\hfill\\
We denote by $\mathbb{N}^*$ the set of positive integers and $\mathbb{N}=\mathbb{N}^* \cup
\{0\}$. We also set $[\![1, n]\!]=\{k\in \mathbb{N} \, \colon\, 1\leq k\leq n\}$
for $n\in \mathbb{N}^*$. 
For $a,b\in \mathbb{R}$, we write 
$a \wedge b$  for  the minimum between $a$ and $b$ and 
$a \vee b$  for  the maximum between $a$ and $b$. 
\medskip

For a real-valued function defined on a set $\Omega$, its
supremum norm is:
\[
\|f\|_\infty=\sup_{\Omega}   |f|.
\]

For   a   measurable  space   $(\Omega,   \mathcal{F})$,   we  denote   by
$\mathcal{M}_1(\Omega)$   the  set   of   probability  measures   on~$\Omega$.
For
$\mu\in \mathcal{M}_1(\Omega)$ and a real-valued measurable function $f$ defined
on $\Omega$, we  will sometimes denote the integral of  $f$ with respect
to      the      measure       $\mu$,      if      well-defined,      by
$\langle \mu | f\rangle = \int_\Omega f(x )\, \mu(dx)= \int f\, d\mu$.
For a bounded measurable real function $F$ defined on $\mathbb{R}$ 
and a bounded real-valued  measurable function $f$ on $\Omega$,
we define the  real-valued  measurable  function $F_f$ defined for
$\mu\in \mathcal{M}_1(\Omega) $ by:
\begin{equation}
	\label{eq:def-Ff}
	F_f(\mu)=F(\langle \mu | f\rangle ).
\end{equation}

For a metric  space $\Omega$ endowed with its  Borel $\sigma$-field, we
endow $\mathcal{M}_1(\Omega)$ with  the topology of weak convergence. We also
denote   by  $\mathcal{D}(\mathbb{R}_+,   \Omega)$   the   space  of   right-continuous
left-limited (càd-làg)  paths from $\mathbb{R}_+$  to $\Omega$.  This  space is
endowed  with the  Skorokhod   topology  (see  \textit{e.g.}   \cite[Chapter
3]{billingsley99}).  \medskip

In what follows  $(\mathbb{X}, \mathrm{d})$ will denote a  Polish metric space  (complete
and separable). We shall write $\mathcal{M}_1=\mathcal{M}_1(\mathbb{X}\times
\{S,I\})$ (where $\mathbb{X}$ is endowed
with its Borel $\sigma$-field) and  $\mathcal{D}=  \mathcal{D}(\mathbb{R}_+,\mathcal{M}_1)$.

\subsection{Individual-based model}\label{sec:IBM}\hfill\\
We consider a  population of $n\in  \mathbb{N}^*$ individuals indexed
by $i\in  [\![1, n]\!]$ and each  characterized by  a value $x^{(n)}_i  \in \mathbb{X}$
(also  called  feature  in  what   follows)  and 
by an epidemiological state $E^{(n),i}_t\in \{S, I\}$,
that varies over (continuous) time $t\in \mathbb{R}_+$
according to whether the individual  $i$ is susceptible or infected.  We
also  set  $\mathcal{X}^{(n)}=(x^{(n)}_i)_{i\in  [\![1, n]\!]}\in  \mathbb{X}^n$.   To  simplify
notation,  we shall write  $x_i$ and  $E^i_t$ for  $x^{(n)}_i$ and
$E^{(n),i}_t$.  
For a time $t\in \mathbb{R}_+$, the population is represented by
the  empirical probability  measure $\eta_t^{(n)}$  on $\mathbb{X}  \times \{S,  I\}$
defined by:
\begin{equation}
	\label{muteN}
	\eta_t^{(n)}(dx, de) = \frac{1}{n} \sum_{i =1}^{n} \delta_{(x_i, E^i_t)}(dx, de).
\end{equation}

This defines a process $(\eta_t^{(n)})_{t\in \mathbb{R}_+}$ 
taking values in the set $\mathcal{M}_1$ 
of probability measures on $\mathbb{X}\times \{S,I\}$ with initial condition:
\begin{equation}
	\eta_0^{(n)}(dx, de) = \frac{1}{n} \sum_{i =1}^{n} \delta_{(x_i, E^i_0)}(dx, de).
	\label{muteN0}
\end{equation}

Note that the measure $\eta^{I, (n)}_t$
defined in the introduction by Equation~\ref{eq:intro-etaN}
corresponds to the restriction of the measure $\eta_t^{(n)}$
to the subset $\mathbb{X}\times \{I\}$.
Since the individual features are kept fixed in time,
the restriction $\eta^{S, (n)}_t$ of the measure $\eta_t^{(n)}$
to the complementary subset $\mathbb{X}\times \{S\}$
is actually deduced as in the following equation:
\begin{equation}
	\label{eta_ekN_S}
	\eta^{S, (n)}_t(dx) 
	= \frac{1}{n} \sum_{i =1}^{n} \mathbbm{1}_{\lbrace E^i_t = S\rbrace} \delta_{x_i}(dx)
	= \mu^{(n)}(dx) - \eta^{I, (n)}_t(dx),
\end{equation}
where $\mu^{(n)}(dx)$  is the marginal probability  distribution in
$x\in \mathbb{X}$
of $\eta^{(n)}_0(dx, de)$ (and thus of $\eta^{(n)}_t$ for all $t\in \mathbb{R}_+$):
\begin{equation}
	\label{eq:marginal-mu}
	\mu^{(n)}(dx)=\eta^{(n)}_0(dx,\{S,I\})= \frac{1}{n} \sum_{i =1}^{n} \delta_{x_i}(dx).
\end{equation}

Let us underline that the process $\eta^{(n)}$ 
is more suited than $\eta^{I, (n)}$ to describe the convergence
because the empirical distribution $\mu^{(n)}$
needs not be introduced as an additional parameter
(in addition to the initial condition), 
see for example Proposition~\ref{prop:micro}.

\medskip

The  heterogeneity  between  individual   contacts  is  modelled  by  an
undirected random  graph  $G^{(n)}=(V^{(n)},E^{(n)})$ along  which  the  disease is	
transmitted.  The vertices $V^{(n)}=[\![1, n]\!]$ represent the individuals and
we  assume   that there is an edge between   $i$ and $j$, that is $(i,j)\in E^{(n)}$,   with  probability
$w^{(n)}_E(x_i,  x_j)\in [0,1]$, and independently for any choice of $1\leq i<j\leq n$.  The function $w^{(n)}_E$ will be called the
\emph{connection density}. 
For an  undirected graph, the  measurable
function   $w^{(n)}_E$   is   necessarily   symmetric, 
\textit{i.e.},  $w^{(n)}_E(x,y)=w^{(n)}_E(y,x)$ for  all $x,  y\in \mathbb{X}$.   When
$(i,j)$ is an  edge of $G^{(n)}$, we will write  classically $i\sim_{G^{(n)}} j$
or  $i\sim j$  when there  is  no ambiguity.

\begin{rem}[Examples of random graphs]
	\label{rem_Vgraphs}
	As  said  in  the  introduction,  the Polish  metric  space  of  features
	$(\mathbb{X},  \mathrm{d})$ parametrizes  the latent  variable assumed to explain the connections in the social network.  In view of the
	graphons which will  be considered, a classical  choice is $\mathbb{X}=[0,1]$
	with independent uniform random features $\mathcal{X}^{(n)}=(x_i)$
	(the reference measure $\mu$ is then the Lebesgue measure).
	However  other choices  are possible  as well  to cover  several graph
	families.

	\begin{enumerate}[(i)]
		\item
		\textbf{Complete  graph.}  
		For  any choice  of  $\mathbb{X}$ and individual features $\mathcal{X}^{(n)}$,  
		choosing     $w^{(n)}_E\equiv 1$ (that  is, $w^{(n)}_E$ constant equal  to 1) provides
		the complete graph where every pair of vertices is connected. The features $\mathcal{X}^{(n)}$ then simply affect the infection rate in a mean-field setting.
		
		\item\label{it:geom}
		\textbf{Geometric   random  graphs.}  
		Choosing   $\mathbb{X}=[0, 1]^d$, 
		the $(x_i)_{i\le n}$ 
		independently and uniformly distributed over $\mathcal{X}$
		(according to the Lebesgue measure
		which is exactly $\mu$)
		and $w^{(n)}_E(x,y)=\mathbbm{1}_{\lbrace |x-y|<r\rbrace}$,  for  a  constant  $r>0$,  provides  an
		example       of       geometric      random       graphs,       see
		e.g.  \cite{penrose_randomgeometricgraphs}. A  direct generalization
		corresponds    to   $w^{(n)}_E(x,y)=g(|x-y|)$    with   a    measurable
		$[0, 1]$-valued function~$g$.
		\item\label{it:SBM}
		\textbf{Stochastic block models (SBM).}  The SBM with $k\in \mathbb{N}^*$
		classes corresponds to a population divided  in $k$ classes and can be
		represented using a finite state  space $\mathbb{X}=[\![1, k]\!]$. 
		Up to reordering, $\mathcal{X}^{(n)}$ is then simply described by the proportion of individuals in each of the $k$
		classes, proportions that are assumed to converge to a certain discrete probability measure $\mu = (\mu(i))_{i\in [\![1, k]\!]}$.
		The symmetric  function $w^{(n)}_E$ then takes a finite number  of values.   See
		e.g. \cite{abbe}.  The SBM is  a very common model for heterogeneous
		population.  The  particular case $k =  1$, where $w_E$ is  simply a
		constant, yields the Erd\H{o}s-R\'enyi random graph, see \textit{e.g.}
		\cite{bollobas1979,vanderhofstad}.
	\end{enumerate}
	In what  follows, it  is possible  to   consider  $G^{(n)}$   a  directed   graph, 
	with   a  slight modification of  the arguments.  However, for the  sake of  clarity and
	since undirected graphs are realistic for social interactions, 
	in this paper we will focus on the case where $G^{(n)}$  is undirected, and in particular $w^{(n)}_E(x,y)$ is symmetric for all $n$.
\end{rem}

We now describe the dynamics of the epidemics 
along the graph using propagation and recovery:
\begin{enumerate}[(i)]
	\item \textbf{Propagation.} Conditionally on $i$ and $j$ being connected 
	and on $i$ being susceptible and  $j$ infected, 
	the disease can be transmitted along the edge 
	at rate $w^{(n)}_I(x_i,x_j)\geq 0$.
	Each edge transmits the infection independently, 
	so for a given susceptible individual $i$, 
	the rate at which it becomes infected is:
	\[\sum_{j\sim i} w^{(n)}_I(x_i, x_j)\, \mathbbm{1}_{\lbrace E^j_t = I\rbrace}.\]
	When the infection occurs, 
	say at time $t$, 
	the state of the population changes from $\eta^{(n)}_{t-}$ to: 
	\[\eta^{(n)}_t=\eta^{(n)}_{t-}-\frac{1}{n}\delta_{(x_i,S)}+\frac{1}{n}\delta_{(x_i,I)}.\]
	
	\item \textbf{Recovery.} An infected individual $i$ recovers at rate $\gamma(x_i)>0$. 
	When the recovery occurs, say at time $t$, 
	the state of the population changes from 
	$\eta^{(n)}_{t-}$ to:
	\[\eta^{(n)}_t=\eta^{(n)}_{t-}-\frac{1}{n}\delta_{(x_i,I)}+\frac{1}{n}\delta_{(x_i,S)}.\]
\end{enumerate}

In  all  the paper,  we  will  assume  that the  \emph{recovery  rate}
$\gamma^{(n)}$,  \emph{infection  rate}  $w^{(n)}_I$,  and  
\emph{connection density} 
$w^{(n)}_E$ are measurable  functions (respectively from $\mathbb{X}$
to $\mathbb{R}_+$,  $\mathbb{X}^2$ to  $\mathbb{R}_+$ and  $\mathbb{X}^2$ to  $[0, 1]$).
Note that they  depend only on the features of the individuals  but not on
time  nor  on times  of  infection  events.   These functions  are  kept
constant throughout the epidemic, but the global infection and recovery
rates, at the  scale of the population, vary because  the partition into
susceptible and infected individuals changes with time.

\medskip

The above  ingredients  define  a pure-jump Markov  dynamics, see
Section~\ref{section:IBM}.  For
the  process $\eta^{(n)}=(\eta^{(n)}_t)_{t\in \mathbb{R}_+}$ to  be itself  Markov, one
needs an  identifiability property, that is,  to be able to  recover the
state of the individuals at time  $t$ from the measure $\eta^{(n)}_t$.  Let
us stress that we shall not use that the process $\eta^{(n)}$ is Markov nor
the   identifiability   property  to   prove   our   main  result,   see
Theorem~\ref{theo_G1}. Nevertheless, we  believe that the
Markov property of $\eta^{(n)}$ deserves to be stated.

We say the model is \emph{identifiable}  if the elements of the sequence
$(x_i)_{i\in [\![1, n]\!]}$ (with $x_i$ depending  also on $n$) are pairwise
distinct, or  equivalently that  the measure $  \mu^{(n)}$ has  exactly $n$
atoms, all of them with mass $1/n$.  When the model is not identifiable,
there is always a natural way to  obtain an identifiable model by enriching the
space $\mathbb{X}$ into  $\mathbb{X}'=\mathbb{X}\times [0, 1]$:
the  feature of the
individual  $i$ can be  described  for   example  by  either   $(x_i,i/n)$  or
$(x_i,U_i)$,  where $(U_i)_{i\in  \mathbb{N}^*}$ are independent  uniform random
variables and independent of the initial condition.

\medskip

Recall that $\mathcal{D}=  \mathcal{D}(\mathbb{R}_+,\mathcal{M}_1) $ denotes the space of càd-làg functions
taking values in $\mathcal{M}_1=\mathcal{M}_1(\mathbb{X}\times  \{S,I\}))$ endowed with the
Skorokhod   topology.      Recall         also  the            notation
$F_f(\eta)=F(\langle \eta | f\rangle ) $ from~\eqref{eq:def-Ff}.
\begin{prop}[Markov property of the process $\eta^{(n)}$]
	\label{prop:micro}
	Let $n\geq 1$ be fixed and assume the model is
	identifiable. Conditionally on the sequence
	$(x_i)_{i\in [\![1, n]\!]}$ and on $G^{(n)}$, the process
	$\eta^{(n)}=(\eta^{(n)}_t)_{t\in \mathbb{R}_+}$
	is a Markov pure-jump process on $\mathcal{M}_1$, and thus is $\mathcal{D}$-valued. Its
	generator $\mathcal{A}^{(n)}$  is  characterized by the following expression, 
	which applies to any bounded  real-valued measurable function $F$ defined  on $\mathbb{R}$,
	to any bounded real-valued measurable function $f$ on $\mathbb{X}\times \{S, I\}$
	and to any point measure $\eta$  with marginal probability distribution given
	by~\eqref{eq:marginal-mu}:
	\begin{multline}
		\label{def:genAn}
		\mathcal{A}^{(n)} (F_f)\, (\eta)
		=    \sum_{i=1}^n  \mathbbm{1}_{\{E^i=I\}} \gamma^{(n)}(x_i)\, 
		\Big(F_f\big(\eta-\frac{1}{n}\delta_{(x_i,I)}
		+\frac{1}{n}\delta_{(x_i,S)}\big)-F_f(\eta)\Big) \\
		+ \sum_{i=1}^n \mathbbm{1}_{\{E^i=S\}}\sum_{j\sim i}w^{(n)}_I(x_i,x_j)\,
		\mathbbm{1}_{\{E^j=I\}}
		\Big(F_f\big(\eta-\frac{1}{n}\delta_{(x_i,S)}
		+\frac{1}{n}\delta_{(x_i,I)}\big)-F_f(\eta)\Big),  
	\end{multline}
	where $x_i$ and $E^i$ denote the feature and
	epidemiological states of the $i$-th atom of $\eta$.
\end{prop}

\begin{proof}
	The process $\eta^{(n)}$ is of  course constant between jump events given
	by infection  or recovery  events.  The proof  is then  an immediate
	consequence of the description of the epidemic dynamics (see also the
	semi-martingale representation~\eqref{eq:martingalepb} below).  Notice
	that  in  \eqref{def:genAn}  the summation  $\sum_{j\sim  i}$  between
	neighboring pairs $(i, j)$ can indeed  be seen as a functional of the
	probability  measure  $\eta$  with marginal  probability  distribution
	given   by~\eqref{eq:marginal-mu}   thanks  to   the   identifiability
	assumption.  (The two other  summations $\sum_{i=1}^n$ are functionals
	of  the   probability  measure  $\eta$  without   the  identifiability
	assumption.)
	
	Notice  there is  no accumulation  of  jump times  for the  process
	$\eta^{(n)}$  (conditionally on  $\mathcal{X}^{(n)}$) as the jump rates
	$\gamma^{(n)}$ and  $w^{(n)}_I$ are bounded on  $\bar{\mathcal{X}}^{(n)}=\{x_i\, \colon\, i\in [\![1, n]\!]\} $   and $(\bar{\mathcal{X}}^{(n)})^2$, since the  state  space for  the
	process $\eta^{(n)}$  is in fact  reduced to the probability  measures on
	the    finite    set    $\bar{\mathcal{X}}^{(n)}    \times\{S,    I\}$.
\end{proof}

\subsection{Convergence  towards an epidemic spreading on a graphon}\label{sec:cv_largepop-graphon}
\hfill\\
We consider now the case of a large graph, \textit{i.e.} 
when $n$ is large and tends to infinity. 
The network structure 
as well as the dynamics of the epidemic 
are suitably rescaled. 
We  measure the overall propagation rate of the epidemic through the
function $w^{(n)}$ defined on $\mathbb{X}^2$ by: 
\begin{equation}
	\label{def:weN}
	w^{(n)}(x,y)=n\,  w^{(n)}_E(x,y)\,  w^{(n)}_I(x,y),
\end{equation}as described in the introduction.
The function $w^{(n)}$ captures the effect  on the number of contacts (through
$n \,  w^{(n)}_E$) weighted by  the intensity of interactions  (that is,
$w^{(n)}_I$). 

Let  $\eta_0\in  \mathcal{M}_1$  be   a  deterministic  probability  measure  on
$\mathbb{X}\times \{S,I\}$ with marginal $\mu$ on $\mathbb{X}$:
\begin{equation}
	\label{def:eta0}
	\eta_0(dx,de)=\mu(dx)\,\Big(\big(1-u_0(x)\big)\delta_{S}(de)+u_0(x)\delta_I(de)\Big), 
\end{equation}
where  $\mu\in \mathcal{M}_1(\mathbb{X})$  and  $u_0$ is  a  $[0, 1]$-valued  measurable
function defined on $\mathbb{X}$. 
It is elementary to generalize the results when $u_0$ is also  random,
but we prefer to keep the result simple. The probability measure  $\mu$ on $\mathbb{X}$ will capture  the asymptotic feature
distribution of the individuals in the population and will be seen as the limit of $\mu^{(n)}$ introduced in \eqref{eq:marginal-mu}.
Let us set $\mu^{\otimes 2}(dx, dy)  = \mu(dx) \mu(dy)$.

\begin{hyp}[Structural conditions]
	\label{hyp:A1}
	We assume that:
	\begin{enumerate}[(i)]
		
		\item For   $n\in \mathbb{N}  ^*$, the process $(\eta^{(n)}_t)_{t\in \mathbb{R}_+}$ is started
		from the   initial condition $\eta_0^{(n)}$ of  the form \eqref{muteN0}.
		
		\item\label{it:init} The sequence $(\eta_0^{(n)})_{n\in \mathbb{N}^*}$ converges in
		probability to $\eta_0$ (in $\mathcal{M}_1$ endowed with
		the topology of the weak convergence). 
		
		\item\label{it:dcv-w}  The sequence $(w^{(n)})_{n\in \mathbb{N}^*}$ converges
		uniformly to a (non-negative) transmission density kernel  $w$ which is 
		bounded  and $\mu^{\otimes 2}$-a.e.\ continuous. 
		\item\label{it:dcv-gamma} The  sequence $(\gamma^{(n)})_{n\in  \mathbb{N}^*}$ converges uniformly  
		to a  (non-negative) recovery  function $\gamma$  which is bounded   and
		$\mu$-a.e.\ continuous.
	\end{enumerate}
\end{hyp}

\begin{rem}[Transmission kernel $w(x,y) =\beta(x) w_E(x, y) \theta(y)$]
	\label{rem:bwq}
	In \cite[Example 1.3]{DDZ22a}, the authors propose to
	represent the transmission density kernel  $w(x, y)$ from    Assumption
	\ref{hyp:A1}-\ref{it:dcv-w},  as $\beta(x) w_E(x, y) \theta(y)$,
	where  $\beta(x)$ captures the susceptibility of  individuals with feature $x$,
	$\theta(y)$  the infectiousness of  individuals with feature $y$,
	while $w_E(x, y)$ relates to the contact rate between individuals with
	features respectively $x$ and $y$.
	For example, using~\eqref{def:weN}, this corresponds  to the choice of $w^{(n)}_E(x,y)=w_E(x,y)$ 
	and $n\, w^{(n)}_I(x,y)=\beta(x)\theta(y)$.
	In the general form $n \,w^{(n)}_E(x,y) w^{(n)}_I(x,y)$,  
	the factor $n\, w^{(n)}_I(x,y)$ 
	may not only encompass effects that involve each member of the pair separately
	(as the ones captured by $\theta(y)$ and $\beta(x)$), 
	but also effects that are more specific to the interaction
	between these pairs of individuals.
\end{rem}

\begin{rem}[Uniform scaling]
	\label{rem_epsN}
	Among the cases covered by Assumption
	\ref{hyp:A1}-\ref{it:dcv-w},
	it is natural to distinguish the situations where the density of the graph is uniform 
	over the feature interactions. 
	Denoting by $\epsilon_n$ a scaling parameter,
	we consider next the cases  
	where $w^{(n)}_I = \epsilon_n \bar{w}_I$ 
	and $w^{(n)}_E = (n \epsilon_n)^{-1} \bar{w}_E$ 
	for some functions $\bar{w}_E, \bar{w}_I$ on $\mathbb{X}^2$, so that, thanks
	to~\eqref{def:weN}, $w^{(n)}=w$ with $w=\bar{w}_E\bar{w}_I$. 
	Among the different possibilities for the scaling parameter $\epsilon_n$,
	the case where $\epsilon_n = n^{-a+ 1}$ for some $a \in [1, 2]$ 
	appears as the most characteristic
	to describe the level of sparsity of the graph,
	as the total number  of edges in the graph is then of order $n^a$.
	As $n$ gets large, 
	the number of contacts of an individual with feature $x$ is 
	well-described by the quantity $n\, \int_\mathbb{X} w^{(n)}_E(x,y) \mu(dy)$.
	The three interesting regimes are summed up in the following table. 
	
	\medskip

	\begin{center}
		{
			\begin{tabular}{lccccc}
				\toprule
				& $a$ & Nb. of contacts & Conn. density $w^{(n)}_E$ & Inf. rate $w^{(n)}_I$ \\
				\midrule
				Dense graph        & $a=2$          & $n$             &         $1$        &    $n^{-1}$ \\
				Sparse graphs      & $a\in(1,2)$ & $n^{a-1}$       &         $n^{a-2}$    &  $n^{-(a-1)}$ \\
				Very sparse graphs & $a=1$          & $1$            &         $n^{-1}$     & $1$ \\
				\bottomrule                                  
			\end{tabular}
		}
		
	\end{center}
	
	\medskip

	The  most  standard homogeneous  case  corresponds  to  the  complete  graph,  
	corresponding to $w^{(n)}_E\equiv~1$. 
	Note that in the dense and sparse case, the infection rate $w^{(n)}_I$ vanishes in the
	large $n$ limit, but it is not the case in the very sparse case.

	Notice  that the  sequence of  graph realizations  $(G^{(n)})_{n\in \mathbb{N}^*}$
	converges to  the graphon prescribed  by $\bar{w}_E$ only in  the dense
	graphs case.  For  the sparse graphs case, there is  already a limiting
	description of this  sequence of graphs in terms of  the subsampling of
	the graphon  prescribed by  $\bar{w}_E$, by expressing  the graph  as a
	kernel on a functional space (see \cite[Theorem 1]{APSS20}).
	By contrast, in the limiting \emph{very sparse} case, where $a=1$, there is local  convergence of the
	graph towards a graph with finite degrees, whose realization brings an
	additional level of heterogeneity (as introduced in \cite{BS01}
	see  also   \cite[Chap.~3]{Bo16} or \cite[Chap.~2]{vH24}). 
\end{rem}

Aside from the structural condition, we also
consider a crucial condition which ensures the convergence of the
individual stochastic models to the 
same deterministic model as in the mean-field case. 
In order to state it, let us denote by $\mathcal{I}_n(g)$ the empirical average
of a function $g$ over all pairs of features:
\begin{equation}
	\label{def_HI}
	\mathcal{I}_n (g) 
	= \frac{1}{n^2}  \mathbb{E}\Bigg[\sum_{i, j\in [\![1, n]\!]} g\Big(x^{(n)}_i,
	x^{(n)}_j\Big) \Bigg]
	= \mathbb{E}\Big[\int g \, d\mu^{(n)}\otimes d \mu^{(n)}\Big]. 
\end{equation}

\begin{hyp}[On the infection rate]
	\label{hyp:techn}
	The average infection rate vanishes when $n$ goes to infinity:
	\[\lim_{n\rightarrow \infty } \mathcal{I}_n \Big( w^{(n)}_I\wedge 1\Big)=0,
	\quad\text{ with $\mathcal{I}_n$  defined by~\eqref{def_HI}.}
	\]
\end{hyp}

We are now ready to present our  main result. Recall that $\mathcal{D}= \mathcal{D}(\mathbb{R}_+,\mathcal{M}_1)$
denotes    the    space   of    càd-làg    paths    from   $\mathbb{R}_+$    to
$\mathcal{M}_1= \mathcal{M}_1(\mathbb{X}\times \{S,I\})$,  and it is endowed  with the Skorokhod
topology.

\begin{theo}[Convergence of the stochastic individual-based model]
	\label{theo_G1}
	Let $\gamma^{(n)}$, $w^{(n)}_E$, $w^{(n)}_I$, $\gamma$,  $w$  and $\eta^{(n)}_0$ 
	be such that Assumptions \ref{hyp:A1} and~\ref{hyp:techn} are 
	satisfied. 
	
	Then, the sequence $(\eta^{(n)})_{n\in \mathbb{N}  ^*}$ converges in probability in
	the    Skorokhod    space    $\mathcal{D}$,   to    the    continuous    process
	$(\eta_t)_{t\in \mathbb{R}^+}$ with deterministic evolution defined by:
	\begin{equation}
		\label{etaDe}
		\eta_t(dx, de) =\mu(dx) \, ((1-u(t, x))\, \delta_S(de)
		+ u(t,x) \, \delta_I(de)),
	\end{equation}	
	where $\mu$ is the marginal of $\eta_0$ on $\mathbb{X}$, see~\eqref{def:eta0},
	and where $(u(t,x))_{t\in \mathbb{R}_+,x\in \mathbb{X}}$ is the unique solution of the integro-differential equation:
	\begin{equation}
		\label{uT}
		\partial_t u(t, x)= (1-u(t, x))\int_{\mathbb{X}}  u(t, y)\, 
		w(x, y)\, \mu(dy)- \gamma(x) u(t, x),
	\end{equation}
	with initial condition $u(0, \cdot)=u_0$, see~\eqref{def:eta0}. 
\end{theo}

Let us  comment on the limit  process. For $t\in \mathbb{R}_+$  and $x\in \mathbb{X}$,
the value of  $u(t, x)$ represents the probability for  an individual with
feature $x$ to be infected at time $t$.  The non-negative recovery rate function
$\gamma$ gives the rate at which individuals with feature $x$ are recovering.
The  transmission kernel  density function  $w(x, y)$  captures in  this
expression  the  contribution to  the  infection  rate of  (susceptible)
individuals with  feature $x$ due to  individuals with feature $y$,  scaled by the
proportion   of   infected   individuals   with  this   feature,   given   by
$ u(t,  y)\, \mu(dy) $.    
Properties of \eqref{uT}  have  been  studied  in
\cite{DDZ22a}, while here we provide an interpretation of this equation as the large-population limit of a stochastic individual-based model.
We refer to Section~\ref{sec:uniqueness} for further comments. 

\begin{rem}[Examples of admissible uniform scaling]
	\label{rem_nAlpha}
	Consider    the    situation    of   Remark    \ref{rem_epsN}    where
	$w^{(n)}_I          =          \epsilon_n         \bar{w}_I$          and
	$w^{(n)}_E =  (n \epsilon_n)^{-1} \bar{w}_E$ with  some bounded functions
	$\bar{w}_E,  \bar{w}_I$ on  $\mathbb{X}^2$  and a  uniform scaling  parameter
	$\epsilon_n$,  so that Assumption~\ref{hyp:A1}-\ref{it:dcv-w}     holds. 
	Assumption~\ref{hyp:techn} is then restated as $\lim_{n \rightarrow\infty} \epsilon_n = 0$,
	thus covering the cases of dense and sparse (but not very sparse) graphs. 
	
	\medskip
	
	This  is  no longer  the  case  for the  very  sparse  graphs model  where
	$a = 1$  and   $\epsilon_n=1$.   As   we  shall  see   in  Subsection
	\ref{sec_Vsparse}, when the  graph is too sparse,  meaning that $w^{(n)}_I$
	is too  large, we cannot expect  the process $\eta^{(n)}$ to  behave as the
	solution $\eta$ to  the problem \eqref{etaDe}.  Let us  mention that the
	very sparse graphs models encompass for example the configuration models
	with fixed  degree distributions in  large populations, which  have been
	treated specifically in the case  of SIR epidemics in various papers
	(\textit{e.g.}
	\cite{andersson_mathscientist,ballneal,barbourreinert,decreusefonddhersinmoyaltran,jansonluczakwindridge,miller,volz}).
\end{rem}

\begin{rem}[Example of the SBM in the dense graphs setting]
	We cover in particular the dense Stochastic Block Model (SBM) with finite feature space, see Remark~\ref{rem_Vgraphs}-\ref{it:SBM} on the SBM and
	Remark~\ref{rem_nAlpha} for the uniform scaling with $a = 2$ and thus
	$\epsilon_n=1/n$. 
	With the notation therein, 
	any individual with feature $q\in [\![1, k]\!]$ recovers at rate $\gamma^{(n)}(q)=\gamma^q$ and 
	is linked to any individual with feature $r$ 
	with probability $w^{(n)}_E(q,r)= w_E^{q, r}\ge 0$ (independently between the pairs),
	triggering the transmission of the disease to the individuals with feature $q$ at rate $w^{(n)}_I(q,r)=
	w_I^{q, r}/n\ge 0$.
	Under the conditions of Theorem~\ref{theo_G1}, the
	limiting process $(\eta_t)_{t\in \mathbb{R}_+}$ 
	has a density with respect to the probability measure $\mu(dx) = \sum_{q
		\le k} \mu_q \delta_q(dx)$ on $\mathcal{M}_1(\mathbb{X})$, where $\mu_q$ is the
	relative size of the population with feature $q$. 
	The measure $\eta_t$ is thus captured 
	through the proportion $u_t^q$ of infected individuals among the ones with
	feature $q$ at time~$t$, with the vector $(u_t^q)_{q\in [\![1, k]\!]}$ satisfying the following system
	of $k$ ODEs:
	\begin{equation*}
		\partial_t u_t^q 
		= (1-u_t^q)\, \sum_{r=1}^k w_I^{q, r} \, w_E^{q, r} \,
		u_t^r	\, \mu_r - \gamma^q \, u_t^q
		\quad\text{for $t\geq 0$ and $q\in [\![1, k]\!]$}. 
	\end{equation*}
\end{rem}

\begin{rem}[Comparison with  \cite{KHT22}]\label{rem:comment}
	
	In \cite{KHT22}, it   is  clear   that  our   Assumptions~\ref{hyp:A1}  and
	\ref{hyp:techn}   are  satisfied. We also consider more general initial conditions and do not restrict to $\mathbb{X} = [0,  1]$,
	which impacts the regularity conditions that one may consider for the parameter functions.
	Contrary to \cite{KHT22}, our framework allows to consider graphs that  strongly exploit  the geometry  of the  latent space       $\mathbb{X}$,      notably       geometric      graphs,       see
	Remark~\ref{rem_Vgraphs}-\ref{it:geom} where $w(x, y) = g(|x-y|)$ with
	any  function   $g:   \mathbb{R}_+   \mapsto  \mathbb{R}_+$   having   at most countably    many
	discontinuities.

	In \cite{KHT22},  the infection rate  (per edge) takes a  constant value that only depends on the population size $n$, and the recovery rate is constant $\gamma(x)\equiv  1$. This does  not allow the  dependence on the  individual  features,  which  might be  of  practical  interest  and insightful concerning the robustness of the results.

	Lastly, it is assumed in \cite{KHT22} that, in our notation,
	$w^{(n)}_E(x, y)  = \kappa_n  w(x, y)$ for  some sequence  $\kappa_n$,
	so that $w^{(n)}_I(x, y)\equiv 1/ (n\kappa_n)$ and 
	$\mathcal{I}_n(w^{(n)}_I \wedge  1) =[(n\kappa_n)\vee 1]^{-1}$.  
	Their  assumption that
	$\log  (n) /  (n\kappa_n)\rightarrow 0$  can, according to Assumption~\ref{hyp:techn}, be relaxed  to $1/  (n\kappa_n)\rightarrow 0$. Let us observe that,
	based on our Theorem \ref{theo_G1}, 
	$w^{(n)}_I$  (and more  precisely  $\mathcal{I}_n(w^{(n)}_I\wedge 1)$)  is the  right
	quantity to control the convergence. The importance of the crucial condition involving $w^{(n)}_I$ is confirmed by the simulations that we present in Section~\ref{sec:discussion}. \hfill $\Box$
\end{rem}

\section{Proof of  Theorem \ref{theo_G1}}\label{sec:proofs}

The proof  of Theorem \ref{theo_G1}  follows a classical  scheme. We 
establish the  uniform tightness of  the distributions of  the processes
$(\eta^{(n)})_{n\in  \mathbb{N}^*}$  in  Section \ref{sec_tight}, see Proposition~\ref{pr_tG1}. By
Prohorov's theorem  \cite[Theorems 5.1  and 5.2  p.59-60]{billingsley99},
this implies that
this sequence of distributions is  relatively compact. We then show
in  Section \ref{sec_cvg}   that  any  potential   limit  must  be  solution   to  a
measure-valued  equation \eqref{eq:eta-sol}, see Proposition~\ref{prop:identify}.  This step  requires a  careful
coupling  between  the  processes   $\eta^{(n)}$  and  auxiliary  processes
$\widetilde{\eta}^{(n)}$   on  the   complete  graph   but  with   modified
infection rate. Using  the uniqueness of the  solution to Equation~\eqref{eq:eta-sol},  see Lemma~\ref{lem:uniqueness},
we  can conclude  that  there  exists a  unique  limiting  value to  the
sequence $(\eta^{(n)})_{n\in \mathbb{N} ^*}$  and hence that it converges to the
unique solution  to Equation~\eqref{eq:eta-sol}, say  $\eta=(\eta_t)_{t\in \mathbb{R}_+}$. From the  uniqueness of
the solution to Equation~\eqref{eq:eta-sol}, we will also establish in Lemma~\ref{lem:uniqueness} that the measures
$\eta_t$       have      densities     with      respect      to
$\mu\otimes (\delta_S+\delta_I)$ that satisfy \eqref{etaDe}-\eqref{uT}.
This will conclude the proof of Theorem \ref{theo_G1}.

\medskip

Before  entering the  proofs of Propositions~\ref{pr_tG1}-\ref{prop:identify} and Lemma~\ref{lem:uniqueness},  we provide  some
results on  the individual-based processes  $\eta^{(n)}$ that will  be used
later.   In  particular,   we  will  rely  heavily   on  the  stochastic
differential equation (SDE) satisfied by these processes.

\subsection{Stochastic differential equation for the individual-based model}\label{section:IBM}
\hfill\\
In this section, we construct the  graphs $G^{(n)}$ on a single probability
space, and the  processes $\eta^{(n)}$, for any $n\in \mathbb{N}^*$,  as solutions of
SDEs driven  by Poisson point measures that
do  not  depend  on  $n$  nor   on  the  graphs  $G^{(n)}$.  This  pathwise
construction allows us on the one  hand to use tools from stochastic
calculus   for   jump   processes   (see   \textit{e.g.}   \cite[Chapter
II]{ikedawatanabe}),  and on  the other  hand to  couple $\eta^{(n)}$  with
$\widetilde{\eta}^{(n)}$ by constructing them on the same probability space
with the same Poisson point measures.
\medskip

\paragraph{Construction of the initial condition of the epidemic process.}
For $n\in \mathbb{N}^*$,
let $\mathcal{X}^{(n)}=(x_i^{(n)})_{i\in [\![1, n]\!]}$ be  a sequence of random variables
taking values in $\mathbb{X}$ and let $\mathcal{E}^{(n)}=(E^{(n),i}_0)_{i\in [\![1, n]\!]}$ be a
sequence  of  random   variables  taking  values  in   $\{S,  I\}$.   For
simplicity,  we  shall   write  $x_i$  and  $E^i_0$   for  $x^{(n)}_i$  and
$E^{(n),i}_0$  when  there  is  no  ambiguity.   The  initial  condition
$\eta^{(n)}_0$ is then given by~\eqref{muteN0}.

\paragraph{Construction of the random graphs $G^{(n)}$.}
Let $\mathcal{V}=(V{(i,  j)})_{1\le i<j}$  be a  family of  independent uniform
random  variables  on   $[0,  1]$,  and  independent   of  $\mathcal{X}^{(n)}$  and
$\mathcal{E}^{(n)}$. For  convenience set $V{(j,i)}=V{(i,j)}$  and $V{(i,i)}=0$.
The graph $G^{(n)}=(V^{(n)}, E^{(n)})$ has vertices $V^{(n)}=[\![1, n]\!]$ and the edge
$(i,j)$ belongs to $E^{(n)}$ if $V{(i,j)}\leq w^{(n)}_E(x_i, x_j)$.

\paragraph{Definition of the Poisson point measures.}

The state  transitions are encoded  thanks to the two  following Poisson
point measures,  that are independent  of $\mathcal{V}$, $\mathcal{X}^{(n)}$  and $\mathcal{E}^{(n)}$.
(For  Poisson point  measures, we  refer  for example  to \cite[Part  I,
Appendix A.2]{ballbrittonlaredopardouxsirltran}).   Let $\mathrm{n}(di)$  be the
counting measure on $\mathbb{N}^*$.
\begin{enumerate}
	\item  For  recovery  events,  we  consider  a  Poisson  point  measure
	$Q_R(ds,  di, du)$  on $\mathbb{R}_+\times  \mathbb{N}^*\times \mathbb{R}_+$,  with intensity
	$ds \, \mathrm{n}(di)\, du$. Each of its  atom, say $(s, i, u)$, is a possible
	recovery event  for the individual $i$  at time $s$.  The  marker $u$
	allows to define whether the possible  event really occurs or not.  On
	the event:
	\begin{equation}
		\label{eq:def-A}
		A^{(n)}(i,u,s)=
		\{i\leq  n\}\cap \{u\le  \gamma^{(n)}(x_i)\}\cap\{E^i_{s-}  = I\},
	\end{equation}
	the
	recovery    of individual $i$  occurs at time $s$,    otherwise    nothing    happens.    By    this
	acceptance-rejection  method, we  can ensure  that recoveries  for $i$
	occur at rate $\gamma^{(n)}(x_i)$ as long as $i$ is infected.
	
	\item  For  infection  events,  we  consider  a  Poisson  point  measure
	$Q_I(ds,            di,            dj,           du)$            on
	$\mathbb{R}_+\times  \mathbb{N}^*\times\mathbb{N}^*\times  \mathbb{R}_+$,  with  intensity
	$ds \,  \mathrm{n}(di)\, \mathrm{n}(dj)\, du  $. Each of  its atom, say $(s, i,
	j, u)$,  is  a possible 
	infection event of individual $i$ by  individual $j$ at time $s$.  The
	marker $u$ serves  for the acceptance-rejection method  to ensure that
	infection  happens  at rate  $w^{(n)}_I(x_i,  x_j)$  as  long as  $j$  is
	infected and $i$ susceptible and connected.  On the event:
	\begin{equation}
		\label{eq:def-B}
		B^{(n)}(i,j,u,s)=  \{i,j\leq  n\} \cap C^{(n)}(i,j,u,s) 
		\cap \{E^i_{s-} = S ; E^j_{s-} = I\},
	\end{equation}
	where
	\begin{equation}
		\label{eq:def-C}
		C^{(n)}(i,j,u,s)= \{ i\sim j\}\cap  \{u\le  w^{(n)}_I
		(x_i, x_j)\}, 
	\end{equation}
	the individual $i$ is infected by $j$, otherwise nothing
	happens. Notice the event $\{i\sim j\}$ can also be written
	$\{V{(i,j)} \leq  w^{(n)}_E(x_i, x_j)\}$.  
\end{enumerate}

To simplify notation, the implicit parameters of the Poisson point
measures are not recalled in the notation and we will use the following
abbreviations:
\[
dQ_R=Q_R(ds,di,du)
\quad\text{and}\quad
dQ_I=Q_I(ds,di, dj,du).
\]

\paragraph{Encoding of $\eta^{(n)}$.}

The evolution of $\eta^{(n)}$ is given by the following equation, for $t\geq 0$:
\begin{multline}
	\label{SDE:etaN}
	\eta^{(n)}_t - \eta^{(n)}_0
	= n^{-1} \int \mathbbm{1}_{\{s<t\}} \, (\delta_{(x_i, S)} - \delta_{(x_i, I)})
	\,\mathbbm{1}_{A^{(n)}(i,u,s)} \,  dQ_R\\
	+ n^{-1} \int \mathbbm{1}_{\{s<t\}}\,   (\delta_{(x_i, I)} - \delta_{(x_i,
		S)}) \, \mathbbm{1}_{B^{(n)}(i,j,u,s)} \,  dQ_I.
\end{multline}

\paragraph{Semi-martingale     decomposition.}      We     denote     by
$(\mathcal{F}^{(n)}_t)_{t\in \mathbb{R}_+}$  the natural filtration associated  to the Poisson
point measures $Q_R$  and $Q_I$, and the random variables  $\mathcal{V}$,
$\mathcal{X}^{(n)}$ and $\mathcal{E}^{(n)}$, so that the  $V_{(i,j)}$'s, the $x_i$'s and  the $E^i_0$'s
are $\mathcal{F}^{(n)}_0$-measurable.   For any  bounded measurable
function  $f$ defined  on $\mathbb{X}\times  \{S, I\}$,  we have  the following
semi-martingale    decomposition    of     the    real-valued    process
$(\langle \eta^{(n)}_t\, \big| \, f\rangle)_{t\in \mathbb R_+}$:
\begin{equation}
	\label{eq:martingalepb}
	\langle \eta^{(n)}_t\, \big| \, f\rangle - \langle \eta^{(n)}_0\, \big| \, f\rangle
	= V^{(n)}_t + M^{(n)}_t,
\end{equation}
where the predictable finite variation process $V^{(n)}=(V^{(n)}_t )_{t\in
	\mathbb{R}_+}$ is given by:
\begin{multline}
	\label{VeNt}
	V^{(n)}_t 
	= \int_0^t ds  \int_{\mathbb{X}} \eta^{(n)}_s(dx, I)\,  \gamma^{(n)}(x) 
	(f(x, S) - f(x, I)) \\
	+ n^{-1} \int_0^t ds  \, \sum_{i\sim j; \, i,j\in [\![1, n]\!]} w^{(n)}_I(x_i,x_j)\,
	\mathbbm{1}_{\lbrace E^i_s=S, \, E^j_s=I\rbrace} 
	\, (f(x, I) - f(x, S)).
\end{multline}
The square integrable martingale process $M^{(n)}=(M^{(n)}_t )_{t\in
	\mathbb{R}_+}$ is given  by:
\begin{multline}
	\label{MeNt}
	M^{(n)}_t 
	= n^{-1} \int \mathbbm{1}_{\lbrace s<t\rbrace}\,  (f(x_i, S) - f(x_i, I)) \mathbbm{1}_{A^{(n)}(i,u,s)}\,
	d\widetilde{Q}_R 		\\
	+ n^{-1}  \int  \mathbbm{1}_{\lbrace s<t\rbrace}\, (f(x_i, I) - f(x_i,
	S))\mathbbm{1}_{B^{(n)}(i,j,u,s)}\,	d\widetilde{Q}_I, 
\end{multline}
where  $\widetilde{Q}_R$  and   $\widetilde{Q}_I$  are  the  compensated
measures of $Q_R$ and  $Q_I$ respectively.  Its quadratic  variation
$\langle M^{(n)}\rangle=(\langle M^{(n)}\rangle_t)_{t\in \mathbb R_+}$ is
given by:
\begin{multline}
	\label{MMeNt}
	\langle M^{(n)}\rangle_t
	=n^{-1} \int_0^t ds \int_{\mathbb{X}} \eta^{(n)}_s(dx, I)\,  \gamma^{(n)}(x) \cdot (f(x,
	S) - f(x, I))^2 \\ 
	+ n^{-2}   \int_0^t ds\, 
	\sum_{i\sim j; \, i,j\in [\![1, n]\!]}w^{(n)}_I(x_i, x_j) \mathbbm{1}_{\lbrace E^i_s = S, \, E^j_s = I\rbrace}
	\,(f(x_i, I) - f(x_i, S))^2.
\end{multline}

For the complete graph, the above
expressions can be simplified using~\eqref{def:weN} as follows.
\begin{lem}[The particular case of the complete graph]
	\label{lem:we=1+simple}
	When   $w^{(n)}_E\equiv1$ and using \eqref{def:weN}, 
	Equations~\eqref{VeNt} and~\eqref{MMeNt} become:
	\begin{multline}
		\label{VeNt-complete}
		V^{(n)}_t 
		= \int_0^t ds  \int_{\mathbb{X}} \eta^{(n)}_s(dx, I)\,  \gamma^{(n)}(x) 
		\cdot	(f(x, S) - f(x, I)) \\
		+ \int_0^t ds \int_{\mathbb{X}} \eta^{(n)}_s(dx, S)\int_{\mathbb{X}} \eta^{(n)}_s(dy, I)\, w^{(n)}(x, y) 
		\cdot (f(x, I) - f(x, S)) , 
	\end{multline}
	and
	\begin{multline}
		\label{MMeNt-complete}
		\langle M^{(n)}\rangle_t
		=n^{-1} \int_0^t ds \int_{\mathbb{X}} \eta^{(n)}_s(dx, I)\,  \gamma^{(n)}(x) \cdot (f(x,
		S) - f(x, I))^2 \\ 
		+ n^{-1}  \int_0^t ds \int_{\mathbb{X}} \eta^{(n)}_s(dx, S)\int_{\mathbb{X}} \eta^{(n)}_s(dy, I)\, w^{(n)}(x, y) 
		\cdot (f(x, I) - f(x, S)) ^2.
	\end{multline}
\end{lem}

\subsection{Tightness}
\label{sec_tight}
\hfill\\
The following tightness criterion relies on the semi-martingale
decomposition given in the previous section. 
Recall that  a
sequence of random elements of $\mathcal{D}$ is  C-tight if it is tight with all
the possible limits being a.s.\ continuous.

\begin{prop}[Tightness]
	\label{pr_tG1}
	Under Assumption~\ref{hyp:A1},
	the sequence of distributions of the processes $(\eta^{(n)})_{n\in \mathbb{N}^*}$ is 
	C-tight on the Skorokhod space $\mathcal{D}$.
\end{prop}

In order to justify the tightness,
we will have to specify global upper-bounds on the jump rate, 
which we relate to the initial condition $\eta_0^{(n)}$,
and more precisely on its marginal  on $\mathbb{X}$ given by  $\mu^{(n)}(dx) =
\eta_0^{(n)}(dx, \{S, I\})$, see~\eqref{eq:marginal-mu}. 
(Notice that the measure $\mu^{(n)}$ is also the marginal on $\mathbb{X}$ of
$\eta_t^{(n)}$ for all $t\in \mathbb{R}_+$).

We consider the following non-negative  function of $\mathcal{X}^{(n)}$:
\begin{equation}
	\label{beta1Def}
	\mathcal{J}^{(n)}= \frac{1}{n} \sum_{i, j\in [\![1, n]\!]: i\sim  j}
	w^{(n)}_I(x_i, x_j).
\end{equation}
The terms $\mathcal{J}^{(n)}$  are potentially of order $n$ since there can be up to
$n^2$ terms in the sum.  
The next lemma asserts that their distributions
are  tight.  By  Assumption \ref{hyp:A1}-\ref{it:dcv-w},  there exist
a finite positive constant $C_w$ and an integer $n_0 \in \mathbb{N}^*$ large enough such that
$\sup_{n\geq  n_0}\|w^{(n)}\|_\infty\leq
C_w$ and  $\|w\|_\infty\leq
C_w$. Without loss of generality, we can assume that $n_0=1$.

\begin{lem}\label{lem:betatight}
	Under   Assumption   \ref{hyp:A1}-\ref{it:dcv-w},  we have:
	\begin{equation}
		\label{eq:EJ-bound}
		\mathbb{E}[\mathcal{J} ^{(n)}]
		\le C_w
	\end{equation}
	and 
	the   sequence   of
	distributions of $(\mathcal{J}^{(n)})_{n\in \mathbb{N}^*}$ is  tight on $\mathbb{R}$.
\end{lem}

\begin{proof} 
	According to the definition of the random graph of
	interactions, we have $i\sim j$  if $V{(i, j)} \le w^{(n)}_E(x_i, x_j)$,
	where the $V{(i, j)}$ are sampled independently.  Thus, we get:
	\[
	\mathbb{E}\Big[\mathcal{J} ^{(n)}\, |\, \mathcal{X}^{(n)}\Big]
	=\frac{1}{n} \sum_{i, j\le n} w^{(n)}_I(x_i, x_j) 
	\mathbb{P}\big(V{(i, j)} \le w^{(n)}_E(x_i, x_j)\big)
	= \frac{1}{n^2}\sum_{i, j\le n} w^{(n)}(x_i, x_j)
	\le C_w.
	\]
	Thanks to the Markov inequality since $\mathcal{J}^{(n)}$ is non-negative, we deduce
	that  $  \mathbb{P}\big(|\mathcal{J}^{(n)} |>C_w/\varepsilon\big)<\varepsilon$  for
	all $\varepsilon>0$ and $n\in \mathbb{N}^*$. This gives the result. 
\end{proof}

We  are now  ready  to  prove Proposition  \ref{pr_tG1}. 

\begin{proof}[Proof of Proposition \ref{pr_tG1}]
	Let
	$D_0$ be  a fixed  set of  real-valued continuous  bounded functions
	defined  on  $\mathbb{X}\times \{S, I\}$,  containing  the  constant  functions,
	and assume that $D_0$ is separating, that is:
	\[ \Big(
	\forall f\in D_0, \;
	\forall \nu, \nu'\in  \mathcal{M}_1,
	\langle \nu, f \rangle =  \langle \nu', f \rangle
	\Big)     \;  \implies \; \nu=\nu'.
	\]
	According to  Theorem II.4.1 in
	\cite{perkins}, the sequence of processes
	$(\eta^{(n)})_{n\in \mathbb{N}^*}$ is C-tight  in $\mathcal{D}$ if and only if:
	\begin{enumerate}[(a)]
		\item\label{it:CCC} \textbf{Compact Containment Condition (CCC).} For all
		$\varepsilon>0$, $T>0$, there exists a compact set $K_{T, \varepsilon}$
		in $\mathbb{X}\times \{S, I\}$ such that:
		\[
		\sup_{n\in \mathbb{N} ^*} \mathbb{P}\Bigg(\sup_{t\leq  T} \eta^{(n)}_t (K_{T,
			\varepsilon}^c) >\varepsilon\Bigg)<\varepsilon.
		\]
		\item  \label{it:proj}  \textbf{Tightness  of the  projections.}
		For all $f\in  D_0$, the sequence
		$(\langle \eta^{(n)}_{\,\,\,\cdot},  f \rangle)_{n\in \mathbb{N}^*}$  is C-tight
		in $\mathcal{D}(\mathbb{R}_+, \mathbb{R})$. 
	\end{enumerate}
	\medskip
	
	We  first  prove  the compactness property~\ref{it:CCC}.   According  to
	Assumption~\ref{hyp:A1}-\ref{it:init},      the     random      sequence
	$(\eta^{(n)}_0)_{n\in  \mathbb{N}^*}$ converges  in  distribution  to $\eta_0$  (in
	$\mathcal{M}_1$  endowed  with the  topology  of  the  weak
	convergence).  Recall $\mu^{(n)}$  is the marginal of  $\eta^{(n)}_0$ on $\mathbb{X}$,
	see~\eqref{eq:marginal-mu};  and $\mu$  is the  marginal of  $\eta_0$ on
	$\mathbb{X}$,  see~\eqref{def:eta0}.    We  deduce  that  the   random  sequence
	$(\mu^{(n)})_{n\in   \mathbb{N}^*}$  converges weakly and in  distribution   to  $\mu$   (in
	$\mathcal{M}_1(\mathbb{X})$).  Since  $\mathcal{M}_1(\mathbb{X})$ is a Polish  space, Prohorov's theorem
	implies  that this  sequence is  tight,  and thus, for all $\varepsilon>0$, there
	is a compact set $K_{\varepsilon}$ in $\mathbb{X}$ such that:
	\[
	\sup_{n\in \mathbb{N}^*} \mathbb{P}\Bigg(\mu^{(n)}(K_{\varepsilon}^c) > \varepsilon\Bigg)<\varepsilon.
	\]
	Since the marginal on $\mathbb{X}$ of $\eta^{(n)}_t$ is also $\mu^{(n)}$, we deduce
	\ref{it:CCC} with the compact set $K_{\varepsilon}\times\{S,
	I\}$:
	\[
	\sup_{n\in \mathbb{N}^*} \mathbb{P}\Big(  \eta ^{(n)}_t (K_{\varepsilon} ^c \times \{S,I\}) > \varepsilon\text{ for
		all $t\in \mathbb{R}_+$} \Big)<\varepsilon.
	\]
	
	We  now prove  \ref{it:proj} on  the tightness  of the  projections.  By
	Assumption  \ref{hyp:A1}-\ref{it:dcv-gamma},   there  exists   a  finite
	positive constant $C_\gamma$  and $n_0 \in \mathbb{N}^*$ large  enough such that
	$\sup_{n\geq   n_0}\|\gamma^{(n)}\|_\infty\leq   C_\gamma$   and 
	$\|\gamma\|_\infty\leq  C_\gamma$. Without loss of  generality, we can
	assume that $n_0=1$.
	
	We  consider the
	processes $\tilde V^{(n)}=(\tilde V^{(n)}_t)_{t\in
		\mathbb{R}_+}$ and $\tilde A^{(n)}=(\tilde A^{(n)}_t)_{t\in
		\mathbb{R}_+}$ defined by:
	\begin{equation}
		\label{eq:def-V-A}
		\tilde V^{(n)}_t=2\|f\|_\infty (C_\gamma+ \mathcal{J}^{(n)})\, t
		\quad\text{and}\quad
		\tilde A^{(n)}_t=\frac{4}{n}\|f\|_\infty^2 (C_\gamma+ \mathcal{J}^{(n)})\,t . 
	\end{equation}
	Thanks  to  Lemma~\ref{lem:betatight},  we   deduce  that the sequences  of
	processes $(\tilde  V^{(n)})_{n\in \mathbb{N}^*}$  and $(\tilde  A^{(n)})_{n\in \mathbb{N}^*}$
	are C-tight in $\mathcal{D}(\mathbb{R}_+, \mathbb{R})$ (as  they are tight with all the possible
	limits being  a.s.  continuous).   Following \cite[Section~VI.3]{jacod},
	we say that a non-negative  non-decreasing, right-continuous and null at
	$0$ process is an ``increasing  process''.  Let $\mbox{Var} (V^{(n)})$ denote the
	variation process of $V^{(n)}$ defined in~\eqref{VeNt}.  Since $\eta^{(n)}$ is
	a    probability     measure,    we    deduce    that     the    process
	$\tilde   V^{(n)}-  \mbox{Var}   (V^{(n)})$  is   an  increasing   process.   Thanks
	to~\eqref{MMeNt}, the  process $\tilde A^{(n)}-  \langle M^{(n)}\rangle$ is  also an
	increasing process.  We deduce from Propositions~VI.3.35 in \cite{jacod}
	that    the    sequences    $(    \mbox{Var}    (V^{(n)}))_{n\in    \mathbb{N}^*}$    and
	$(\langle  M^{(n)}\rangle)_{n\in  \mathbb{N}^*}$  are  C-tight.   Then,  we  deduce  from
	Propositions~VI.3.36 therein that the  sequence $( V^{(n)})_{n\in \mathbb{N}^*}$ is
	C-tight.  Since  the sequence $(\eta^{(n)}_0)_{n\in \mathbb{N}^*}$  is tight thanks
	to  Assumption~\ref{hyp:A1}-\ref{it:init},   we  also   deduce  from
	Theorem~VI.4.13         therein         that        the         sequence
	$(\eta^{(n)}_0+ M^{(n)})_{n\in  \mathbb{N}^*}$ is  tight.  Then  use Corollary~VI.3.33
	therein        to         deduce        that         the        sequence
	$(\eta^{(n)}=\eta^{(n)}_0+ V^{(n)}+ M^{(n)})_{n\in \mathbb{N}^*}$ is tight.
	
	Eventually,  notice  that  the   sequence  $(\tilde  A^{(n)})_{n\in  \mathbb{N}^*}$
	converges    to     $0$,    which    implies    that     the    sequence
	$(\langle   M^{(n)}\rangle)_{n\in   \mathbb{N}^*}$    and   thus   the   sequence
	$(M^{(n)})_{n\in \mathbb{N}^*}$ converge also to 0  a.s. and in $L^1$ uniformly for
	$t\in [0,  T]$, $T$ finite. Since  $( V^{(n)})_{n\in \mathbb{N}^*}$ is  C-tight, we
	also  deduce that  $(\langle \eta^{(n)}_{\,\,\,\cdot},  f \rangle)_{n\in \mathbb{N}^*}$   is C-tight.   Then take  for
	$D_0$ the set of bounded continuous measurable  functions defined on $\mathbb{X}\times\{S, I\}$ to deduce
	that \ref{it:proj} holds.
	\medskip
	
	We deduce from 
	\cite[Theorem II.4.1]{perkins} that the sequence 
	$(\eta^{(n)})_{n\in \mathbb{N}^*}$ is C-tight  in~$\mathcal{D}$. 
\end{proof}

As a by-product of the above proof, we get the
following result.

\begin{cor}
	\label{cor:cv-M-0}
	Suppose  the assumptions of Proposition~\ref{pr_tG1} hold.
	Let  $f$ be a  bounded measurable
	function   defined  on $\mathbb{X}\times  \{S, I\}$. Then,
	the sequence of
	martingales $(M^{(n)})_{n\in \mathbb{N} ^*}$ defined by~\eqref{MeNt}, is such that for all
	$T\in \mathbb{R}_+$:
	\[
	\lim_{n\rightarrow \infty }     \mathbb{E}\Bigg[\sup_{t\in [0, T]}
	\Big(M^{(n)}_t\Big)^2\Bigg]=0.
	\]
\end{cor}
\begin{proof}
	At  the end  of the  proof  of Proposition~\ref{pr_tG1},  we get  that
	$  \langle M^{(n)}\rangle_t  \leq \tilde  A^{(n)}_t$,  with $\tilde A^{(n)}$  defined
	in~\eqref{eq:def-V-A}. Then  use  Doob's inequality  for  the  square
	integrable martingale $M^{(n)}$ and~\eqref{eq:EJ-bound} to conclude.
\end{proof}

\subsection{Identification of the equation characterizing the limiting values}
\label{sec_cvg}\hfill\\
Under the hypotheses of Proposition~\ref{pr_tG1}, the sequence $(\eta^{(n)})_{n\in \mathbb{N}^*}$
is C-tight in $\mathcal{D}$. Let $\bar \eta=(\bar{\eta}_t)_{t\in \mathbb{R}_+}\in
\mathcal{D}$ be a possible limit in distribution. In particular, the function $t\mapsto \bar \eta _t$ is continuous.

By construction,  we  have that  $\eta^{(n)}_t(dx \times \{S, I\})$  is  equal to  $\mu^{(n)}$,
see~\eqref{eq:marginal-mu}   for   all   $t\in  \mathbb{R}_+$.   According   to
Assumption~\ref{hyp:A1}-\ref{it:init},           the           sequence
$(\mu^{(n)})_{n\in \mathbb{N}^*}$  converges in distribution to  $\mu$.  Since the
map  $\nu   \mapsto  \nu(dx\times \{S,I\})$  from  $\mathcal{M}_1$  to
$\mathcal{M}_1(\mathbb{X})$  is continuous,  and  since $\bar  \eta$  is continuous  in
$\mathcal{D}$, we deduce that a.s.:
\begin{equation}
	\label{eq:densite-eta-bar}
	\bar \eta_t(dx\times \{S,I\})=\mu(dx) 
	\quad\text{for all $t\in \mathbb{R}_+$}. 
\end{equation}

We state in the next proposition that $\bar{\eta}\in \mathcal{D}$ is a (continuous) solution to:
\begin{equation}
	\label{eq:eta-sol}
	\Psi_{f,t}(\nu)=0
	\quad\text{for all $t\in \mathbb{R}_+$},
\end{equation}
for any real-valued continuous bounded function $f$ defined on
$\mathbb{X}\times\{S,I\}$, where for $\nu=(\nu_t)_{t\in \mathbb{R}_+}\in \mathcal{D}$:
\begin{multline}
	\label{eq:def-Psi}
	\Psi_{f,t}(\nu)
	=	\langle \nu_t- \nu_0\, \big| \, f\rangle 
	- \int_0^t dr\,  \int_{\mathbb{X}} \nu_r(dx, I)\, \gamma(x)\cdot  (f(x,  S) - f(x, I))\\
	- \int_0^t dr\, \int_{\mathbb{X}} \nu_r(dx, S) \int_{\mathbb{X}} \nu_r(dy, I)
	w(x, y) \cdot (f(x, I) - f(x, S)).
\end{multline}

\begin{prop}[Property of the limiting process]
	\label{prop:identify}
	Under    Assumptions~\ref{hyp:A1}   and~\ref{hyp:techn},    any
	limiting  value $\bar{\eta}$  of $(\eta^{(n)})_{n\in  \mathbb{N}^*}$ in  $\mathcal{D}$ is
	a.s.\ continuous and solution of Equation \eqref{eq:eta-sol}.
\end{prop}

The rest of the section is devoted to the proof of Proposition
\ref{prop:identify}.   It  is  divided  into several  steps:  first,  we
consider the case  of a complete graph with  $w^{(n)}_E\equiv1$, second, in
the heterogeneous  case, we  construct a  coupling between  the epidemic
process $\eta^{(n)}$ and an epidemic process $\widetilde{\eta}^{(n)}$ defined on the complete graph,
but with a modified infection rate, third we end the proof.

\subsubsection*{Step 1: Limiting equation for the epidemic on the complete
	graph ($w^{(n)}_E\equiv 1$)}

In this section, we assume that  $w^{(n)}_E\equiv 1$. Thanks to
Lemma~\ref{lem:we=1+simple}, we get from \eqref{eq:martingalepb} and
\eqref{VeNt-complete} that:
\begin{multline*}
	\Psi_{f,t}(\eta^{(n)})=  M^{(n)}_t
	+ \int_0^t dr \,  \int_{\mathbb{X}}  \eta^{(n)}_r(dx, S) 
	\int_{\mathbb{X}} \eta^{(n)}_r(dy, I)\;
	(w^{(n)}(x,y)-w(x, y))\cdot 
	(f(x, I) - f(x, S))\\
	+ \int_0^t dr \, \int_{\mathbb{X}}  \eta^{(n)}_r(dx, I)\;
	(\gamma^{(n)}(x)-\gamma(x))\cdot  (f(x,  S) - f(x, I)).
\end{multline*}
By Assumption \ref{hyp:A1}-\ref{it:dcv-w},\ref{it:dcv-gamma}, the following upper-bound holds:
\[
|\Psi_{f,t}(\eta^{(n)})|\leq 
|M^{(n)}_t|+ 2t \|f\|_\infty \big(\|w^{(n)}-w\|_\infty +
\|\gamma^{(n)}-\gamma\|_\infty\big).
\]
Then use Corollary~\ref{cor:cv-M-0} and the uniform convergence of
$\gamma^{(n)}$ and $w^{(n)}$ to $\gamma$ and $w$ to conclude that
$\lim_{n\rightarrow \infty } \mathbb{E}\Big[ |\Psi_{f,t}(\eta^{(n)})| \Big]=0$. 
\medskip

Since    the    functions    $w$    and    $\gamma$    are    continuous
$\mu^{\otimes 2}$-a.s.\  and $\mu$-a.s.\,  we deduce  that the  functional
$\Psi_{f,t}$ is continuous at any $\nu\in \mathcal{D}$ such that $\nu_s(dx\times \{S,I\})$ is
absolutely  continuous  with  respect  to   $\mu$  for  a.e.  $s\in  [0,
t]$. Thanks  to~\eqref{eq:densite-eta-bar}, a.s.\ for all  $t\geq 0$ the
probability  measure  $\bar\eta_t(dx\times \{S,I\})$  is absolutely  continuous  with
respect  to $\mu(dx)$.  We  deduce that  $\Psi_{f,t}(\eta^{(n)})$ converges  in
distribution     to      $\Psi_{f,t}(\bar     \eta)$,      and     since
$|\Psi_{f,t}(\eta^{(n)})|\leq 2\|f\|_\infty\cdot \Big(1+t C_w  + t C_\gamma\Big)$, see~\eqref{eq:def-Psi},
we                              deduce                              that
$   \lim_{n\rightarrow    \infty   }    \mathbb{E}\Big[   |\Psi_{f,t}(\eta^{(n)})|
\Big]=\mathbb{E}\Big[    |\Psi_{f,t}(\bar    \eta)|   \Big]$.
This concludes that
$ \mathbb{E}\Big[ |\Psi_{f,t}(\bar  \eta)| \Big]=0$, that is,  $\bar \eta$ is
a.s.\ solution of~\eqref{eq:eta-sol}.

\subsubsection*{Step 2: Coupling for heterogeneous contact graphs  ($w^{(n)}_E\neq 1$)}

Now  we consider  the general  case  of the  epidemic process  $\eta^{(n)}$
associated to  connection density $w^{(n)}_E$, infection  rate $w^{(n)}_I$
and recovery  rate $\gamma^{(n)}$.  We  shall couple the  process $\eta^{(n)}$
with an epidemic process $\widetilde{\eta}^{(n)}$  on a complete graph associated
to  connection density  $\widetilde{w}^{(n)}_E\equiv  1$,  infection  rate
$\widetilde{w}^{(n)}_I=w^{(n)}_I     w^{(n)}_E$     and      same     recovery     rate
$\widetilde{\gamma}^{(n)}=\gamma^{(n)}$. The two epidemic processes shall also start
with the  same initial condition $\eta^{(n)}_0=\widetilde{\eta}^{(n)}_0$.   The next
proposition,     whose    technical     proof     is    postponed     to
Section~\ref{sec:appendix}, gives  for a  well chosen coupling  an upper
bound on the distance in variation between the two epidemic processes.

We denote by $\|\nu\|_{TV}$  the total variation norm of a signed measure
on $\mathbb{X}\times\{S, I\}$.
Recall the
functional $\mathcal{I}_n$ in~\eqref{def_HI}.

\begin{prop}[A control using coupling]
	\label{pr_cpl_G1}
	Under Assumption~\ref{hyp:A1}, for  all $T\geq 0$, there  exists a finite
	constant    $C_T$ (independent of $n$)   and    a    sequence    of    epidemic    processes
	$(\widetilde{\eta}^{(n)})_{n\in \mathbb{N}^*}$ on a complete graph (with $\widetilde{\eta}^{(n)}$
	associated to  connection density  $\widetilde{w}^{(n)}_E\equiv  1$, infection
	rate       $\widetilde{w}^{(n)}_I=w^{(n)}_I        w^{(n)}_E$,        recovery       rate
	$\widetilde{\gamma}^{(n)}=\gamma^{(n)}$         and        initial         condition
	$\widetilde{\eta}^{(n)}_0= \eta^{(n)}_0$) such that:
	\begin{equation*}
		\mathbb{E}\Bigg[\sup_{t\in [0,T]} \|\eta^{(n)}_t - \widetilde{\eta}^{(n)}_t\|_{TV}\Bigg] 
		\le C_T\,  \mathcal{I}_n(w^{(n)}_I\wedge 1).
	\end{equation*}
\end{prop}

We shortly
explain how the auxiliary  process $\widetilde{\eta}^{(n)}$ is constructed.  The
graph $G^{(n)}$ on  which the epidemic described by  $\eta^{(n)}$ spreads is
encoded   by  the   random  variables   $(V{(i,j)})_{1\leq  i<j}$   of
Section~\ref{section:IBM}.

The dynamics of $\widetilde{\eta}^{(n)}$ can be interpreted as the one for which
this graph  structure is initially  the same  (that is, encoded  by the
random variables $(V{(i,j)})_{1\leq i<j}$) but then each edge involved
in  an infection  event  is  resampled. This  amounts  to  modifying  the
infection rate by multiplying it by  the probability to get the edge,
that is $\widetilde{w}^{(n)}_I=w^{(n)}_I w^{(n)}_E$, and considering that the propagation
holds on a complete graph.

\subsubsection*{Step 3: End of the proof}

Suppose  that the  sequence  $(\eta^{(n)})_{n\in \mathbb{N}^*}$,  which is  C-tight
according to Proposition~\ref{pr_tG1}, converges in distribution along a
sub-sequence   $(n_k)_{k\in   \mathbb{N}^*}$   towards  a   certain   trajectory
$\eta\in  \mathcal{D}$.    Thanks  to  Proposition~\ref{pr_cpl_G1} and Assumption~\ref{hyp:techn},   the  sequence
$(\widetilde{\eta}^{(n_k)})_{k\in  \mathbb{N}^*}$  (of  epidemic process  on  complete
graphs) converges also  in distribution towards $\eta$.   By  Step 1, the
limit $\eta$  is solution to Equation~\eqref{eq:eta-sol}.  Thus, all the
limiting values of  the sequence $(\eta^{(n)})_{n\in \mathbb{N}^*}$  are solutions of
Equation~\eqref{eq:eta-sol}. This ends the proof of
Proposition~\ref{prop:identify}. 

\subsection{Uniqueness of the limit}
\label{sec:uniqueness}\hfill\\
We  shall   prove  the   existence  and   uniqueness  of   the  solution
of~\eqref{eq:eta-sol}.   The measurable  non-negative   functions $w$  and
$\gamma$  defined respectively  on $\mathbb{X}^2$  and  $\mathbb{X}$ are  assumed to  be
bounded.  We first  comment on the existence and  uniqueness of solution
to the ODE~\eqref{uT} in the  Banach space $(\mathcal{L}^\infty , \|\cdot\|_\infty$
of  bounded  measurable  functions  defined on  $\mathbb{X}$  endowed  with  the
supremum  norm.   Let   $\Delta\subset  \mathcal{L}^\infty  $  be   the  set  of
non-negative functions bounded  by 1.  Under the  further assumption that
$\gamma$ is positive, it is  proved in \cite[Section~2]{DDZ22a} that for
an initial condition  $u_0\in \Delta$, the  ODE~\eqref{uT} admits a
unique (maximal) solution;  it is defined on $\mathbb{R}_+$ and  takes values in
$\Delta$.   A careful  reading of  \cite{DDZ22a} provides  that one  can
assume that $\gamma$  is only non-negative without  altering the results
nor   the   proofs   of   Section  2   therein.    Thus   the   solution
$(u_t=(u(t,x))_{x\in   \mathbb{X}})_{t\in  \mathbb{R}_+}$   of~\eqref{uT}  with   initial
condition $u_0\in  \Delta$ exists  and is  unique.  

The continuous
$\mathcal{M}_1$-valued  process $\eta=(\eta_t)_{t\in  \mathbb{R}^+}$,
with  deterministic evolution,  defined  by~\eqref{etaDe}  is clearly  a
solution of~\eqref{eq:eta-sol}  for any continuous function  $f$ defined
on   $\mathbb{X}\times\{S,I\}$    with   initial   condition    $\eta_0$   given
by~\eqref{def:eta0}.   We  shall  now  prove  the  following  uniqueness
result.

\begin{lem}[Uniqueness for Equation~\eqref{eq:eta-sol}]
	\label{lem:uniqueness}
	Assume the  measurable non-negative  functions $w$  and $\gamma$ defined
	respectively  on $\mathbb{X}^2$  and $\mathbb{X}$  are bounded.   If
	$\bar   \eta\in  \mathcal{D}$   is   a.s.\  a   continuous  process   solution
	of~\eqref{eq:eta-sol}  (for  any  real-valued continuous bounded   function  $f$
	defined on  $\mathbb{X}\times\{S,I\}$) with initial condition  $\eta_0$ given
	by~\eqref{def:eta0}, then a.s.\ $\bar \eta=\eta$.
\end{lem}

\begin{proof}
	Let  $\bar \eta\in \mathcal{D}$  be continuous and a  solution of~\eqref{eq:eta-sol}  with initial
	condition $\eta_0$  given by~\eqref{def:eta0}.  Taking  $f$ such that
	$f(\cdot, S)=f(\cdot, I)$, we deduce from~\eqref{eq:eta-sol} that for
	all $t\in \mathbb{R}_+$ the
	marginals of $\bar \eta_t$ on $\mathbb{X}$ are constant equal to the
	marginal of $\eta_0$ given by $\mu$ (see~\eqref{def:eta0}). This
	implies that for all $t\in \mathbb{R}_+$, the probability measure $\bar \eta_t$ is
	absolutely continuous with respect to $\mu\otimes (\delta_S+
	\delta_I)$, and thus:
	\[
	\bar \eta(dx, de)
	=\mu(dx) \, ((1-\bar u(t, x))\, \delta_S(de)
	+ \bar u(t,x) \, \delta_I(de)),
	\]
	for                some               measurable                function
	$\bar  u=(\bar u_t=(u(t,x))_{x\in  \mathbb{X}})_{t\in \mathbb{R}_+}$  taking values  in
	$[0, 1]$.  Let $g$ be  a real-valued  continuous bounded  function defined  on $\mathbb{X}$
	and $t\in \mathbb{R}_+$.
	Taking       $f(x,e)=g(x)\mathbbm{1}_{\{e=I\}}$,      we       deduce      from
	$ \Psi_{f,t}(\bar \eta)=0$ that $\int_\mathbb{X} g(x)\,\bar U_t(x)\, \mu(dx)=0$ where:
	\[
	\bar U_t(x)=\bar u_t(x) -u_0(x)  
	-  \int_0^t dr\, \Big(
	(1- \bar  u_r(x)) \, \int_\mathbb{X} \mu(dy)\;  w(x,y)\cdot \bar u_r(y) 
	-  \gamma(x)\cdot \bar u_r(x)\Big).
	\]
	Since $g$ and $t$ are arbitrary, we deduce that $\bar U_t=0$ $\mu$-a.e.\ for
	all $t\in \mathbb{R}_+$. 
	Using~\eqref{uT}, we obtain that for all $t\in \mathbb{R}_+$:
	\[
	\|\bar u_t - u_t\|_1 \leq  (2 C_w+ C_\gamma) \int_0^t  dr\, \|\bar
	u_r - u_r\|_1,
	\]
	where $\|\cdot\|_1$ denotes the usual norm on $L^1(\mathbb{X}, \mu)$.
	By Grönwall's Lemma, we deduce that $\|\bar u_t - u_t\|_1=0$ for all
	$t\in \mathbb{R}_+$. This implies that for all $t\in \mathbb{R}_+$, $\mu$-a.e.\ $\bar
	u_t=u_t$, that is $\bar \eta_t=\eta_t$ for all $t\in \mathbb{R}_+$. 
\end{proof}

\section{Simulations and discussion}\label{sec:discussion}
In this section, we will focus on the special case of an homogeneous population to illustrate our main convergence result and to highlight the crucial role of Assumption~\ref{hyp:techn}.

\subsection{Simulations setup: Erdős–Rényi  graph and homogeneous infection rate}
\hfill\\
We  consider the Erdős–Rényi  graph, which  amounts to
taking $w^{(n)}_E$ homogeneous  over the population, and  thus depending only
on $n$, not on the feature.  
We consider an infection rate $w^{(n)}_I$ which is homogeneous over the population, too,
so  that $w^{(n)}=n\, w^{(n)}_Ew^{(n)}_I$ is homogeneous and converges to a constant transmission kernel $w$.
For simplicity, the initial condition $u_0$ is also taken to be constant.
In this case, Equation~\eqref{eq:intro-uT} 
prescribes the density
(identified with the proportion)
of the infected population,
now uniform over the features. 
We thus recover the classical SIS model introduced
by Kermack and McKendrick 
(in \cite{KM32, KM33}, as a special case): 
\begin{equation}
	\partial _t u(t)=w\cdot(1-u(t))\cdot   u(t) - \gamma\cdot u(t).  
	\label{KM_SIS}
\end{equation}
If $w=0$, then there
is asymptotically no propagation of the epidemics: our  results also cover this degenerate case, which will not be considered further.
Assuming that the  model is  not degenerate,  that is
$w>0$, the most natural choice is setting $w^{(n)}=w$.
For the resulting Erdős–Rényi  graph $G^{(n)}$,
edges in the graph of $n$ vertices 
are thus kept with probability:
\[
w^{(n)}_E = \frac{w}{n\, w^{(n)}_I}\cdot
\]
Assumption~\ref{hyp:techn} corresponds to
$\lim_{n\rightarrow
	\infty }   w^{(n)}_I=0$, that is,  $\lim_{n\rightarrow
	\infty } n\,  w^{(n)}_E=+\infty $. 
Notice that  $n\, w^{(n)}_E$ is the expected number of  neighbors (for
any  individual),   which  is   purely  a   local  quantity.
\begin{rem}[Assumption~\ref{hyp:techn} and connectivity] 
	\label{rem:Erdos-Renyi}   
	We compare the condition  
	$\lim_{n\rightarrow \infty } n\, w^{(n)}_E=+\infty $ 
	with the well-known connectivity properties of Erdős–Rényi graphs, which have a sharp connectivity threshold when $n\, w^{(n)}_E $ is around $\log n$
	(see e.g. \cite[Theorem~5.8]{vanderhofstad}).
	Erdős–Rényi graphs may exhibit the so-called super-critical regime (see e.g. \cite[Section~4.4]{vanderhofstad}): if $n\, w^{(n)}_E\geq c>1$, with a probability
	that converges  to 1  as $n$  tends to infinity, then there exists  a 
	connected  component,  called  the  giant component,  which  contains  a
	positive  fraction  of  the  $n$ vertices, while the second-largest component has at most $O( \log n)$ vertices.  
	Furthermore, if $n\, w^{(n)}_E\geq d \log(n)$,  with a probability that converges to  1 as $n$ tends to  infinity, then there are  only a giant component  and isolated vertices if $d>1/2$, and the graph is connected if $d>1$ (see \cite[Section~5.3, Proposition~5.10 and Theorem~5.8]{vanderhofstad}).
	Instead, in the super-critical  regime with bounded $n\, w^{(n)}_E$, the size of the  second largest component
	is of order  $\log(n)$ (see \cite[Section~4.4]{vanderhofstad}).

	Based on this description, we can observe that Assumption~\ref{hyp:techn} does not imply the graph to be connected.
	However, Assumption~\ref{hyp:techn} implies that the graph is in the super-critical regime. 
	At the same time, the super-critical regime also includes some graphs which we qualify of `very sparse', where $n\, w^{(n)}_E$ is constant and hence  Assumption~\ref{hyp:techn} fails.
	We will consider the latter case in Sect.~\ref{sec_Vsparse}, and show that the mean-field approximation is not good in this regime.
	\hfill $\Box$ \end{rem}
\medskip

For the simulations, we shall consider
the numerical values $w=3$ and $ \gamma=0.7$.
The reproduction number is given by:
\[
R_0=\frac{w}{\gamma}\simeq 4.3.
\]
Since $R_0>1$, we get that $$\lim_{t\rightarrow \infty } u(t)=1
-1/R_0=:u_*$$ and that $u(t)\approx u_* + C \mathrm{e}^{- (R_0 -1)t}$ for large~$t$
and some  finite constant $C$, provided the initial condition
$u(0)$ is positive. In particular, on the time-interval $I_0=[20, 80]$, the
function $u$ is numerically constant and equal to the equilibrium $u_*$.

Let $u^{(n)}(t)$ denote  the  proportion  of  the infected
population  at  time $t\geq 0$, that  is,  $u^{(n)}(t)=\eta^{(n)}_t(\mathbb{X}, I)$  in  the
terminology of Section~\ref{sec:IBM}. 
For the simulations, we use the Gillespie algorithm to sample exact
trajectories (up to numerical limitations), with the 
initial population  completely infected, that is,  $u^{(n)}(0)=1$.

\begin{rem}
	As in general 
	the  asymptotic equilibrium is not unique
	(notably when the graph is not connected),
	starting from the initial condition $u_0 = 1$
	shall provide the estimation on the  asymptotic equilibrium
	that is maximal in terms of the infected proportion of individuals
	(cf \cite{DDZ22a}, provided the stochasticity does not move the process away from it).
	As starting simulations with any positive initial conditions 
	does not change  the results significantly, 
	we shall only consider $u_0=1$ as initial condition.
\end{rem}

In the simulations we will either consider a fixed population size $n_0=2\,  000$,
or show plots for growing values of $n$.
When studying the different sparsity regimes
as in Remark~\ref{rem_epsN}, we will consider
$w^{(n)}_E\asymp n^{a-2}$, which corresponds to $w^{(n)}_I\asymp n^{-a+1}$; for ease of notation,
we will rather use the following parameter instead of~$a$:
\[
\alpha=a-1\geq 0 \quad\text{so that}\quad w^{(n)}_E\asymp n^{\alpha-1}
\quad\text{and}\quad w^{(n)}_I\asymp n^{-\alpha},
\]
in order to focus on
the difference
between the case $\alpha>0$ (Assumption~\ref{hyp:techn} holds) and the very
sparse case $\alpha=0$ (Assumption~\ref{hyp:techn} fails).
More precisely, 
in all plots concerning the asymptotic behavior $w^{(n)}_I \asymp n^{-\alpha}$ we will use the values
\[ w^{(n)}_I= 1.2\,  (n/n_0)^{-\alpha}. \]

\subsection{Numerical observation of the convergence}
\hfill\\
In this section, we illustrate how the solution $u$  of \eqref{KM_SIS}
(or rather its equilibrium value $u_* = \lim_{t \to \infty} u(t)$, which is attained by $u(t)$ within numerical precision for time $ t \ge 20$) 
is a good approximation of $u^{(n)}$ whenever $w^{(n)}_I$ is small enough, consistently with the convergence of $u^{(n)}$ to $u$, for $n \to \infty$, if $w^{(n)}_I \to 0$ (i.e., Assumption~\ref{hyp:techn} holds).

\subsubsection{The proportion $u^{(n)}(t)$  of infected population and the limit $u_*$}
\hfill\\
\begin{figure}
	\centering
	\begin{subfigure}[b]{0.48\textwidth}
		\centering
		\includegraphics[width=\textwidth]{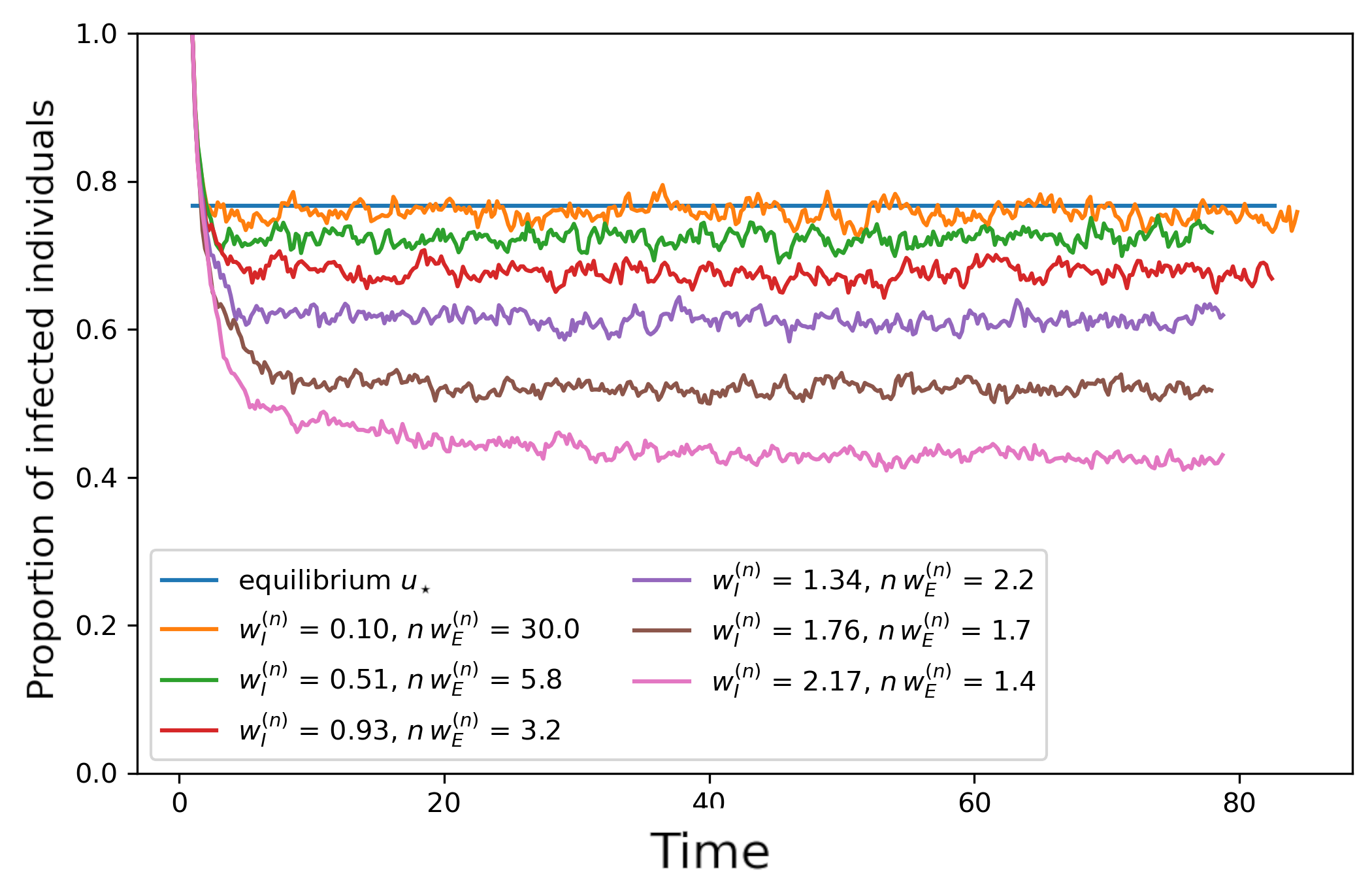}
		\caption{Population size $n=n_0= 2,000$ is fixed: for $t$ larger than~10, $u^{(n)}(t)$ is closer to $u_*= 1-\gamma/w$ when~$w^{(n)}_I$ is smaller.}
	\end{subfigure}
	\hfill
	\vrule
	\hfill
	\begin{subfigure}[b]{0.48\textwidth}
		\centering
		\includegraphics[width=\textwidth]{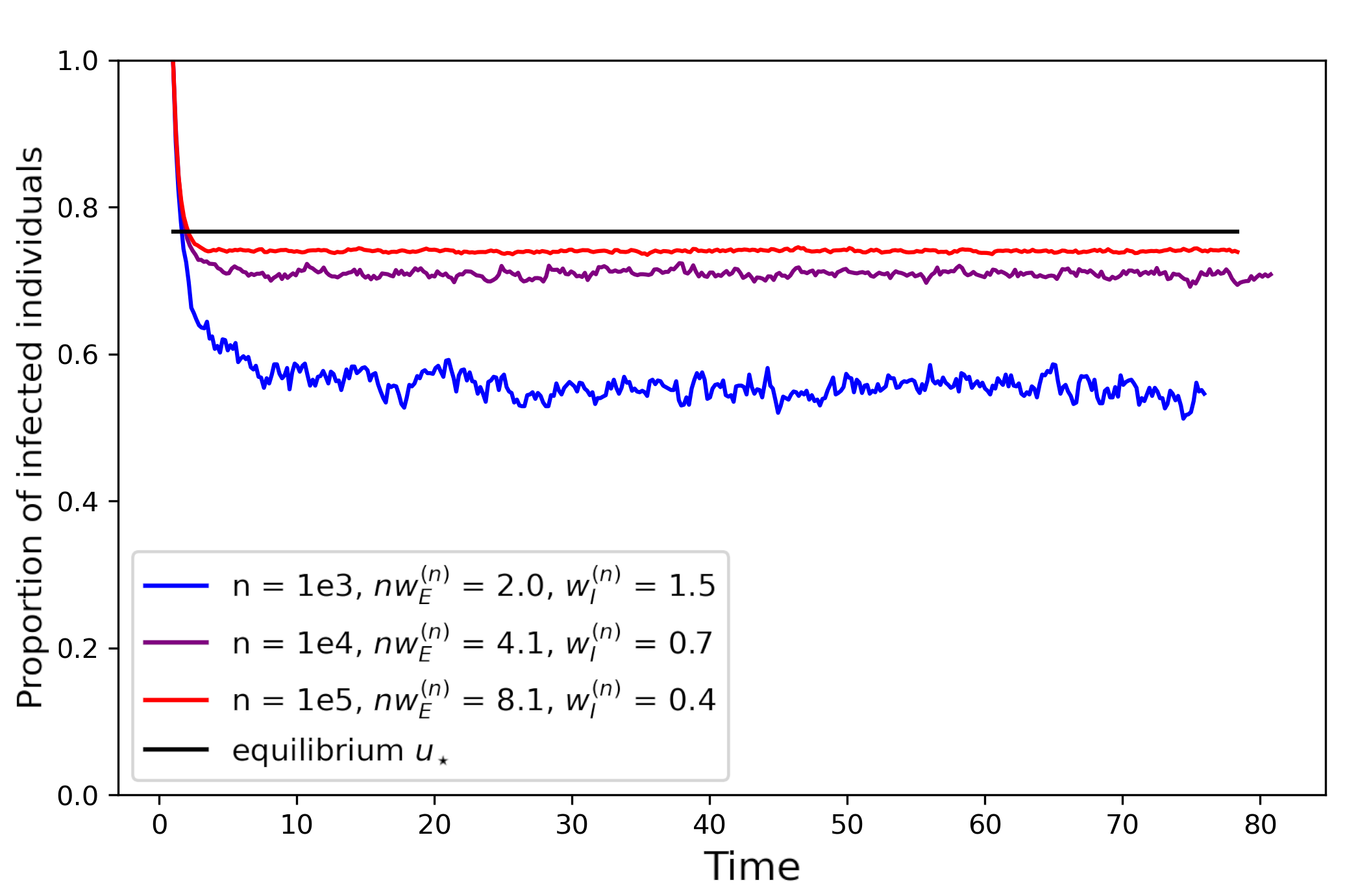}
		\caption{$w^{(n)}_I \asymp n^{-\alpha}$ with   $\alpha = 0.3$, for different values of~$n$.}
	\end{subfigure}
	\caption{Evolution of $t\mapsto u^{(n)}(t)$ for different choices of $w^{(n)}_I$.}
	\label{fig_cvg}
\end{figure}

In  Fig.~\ref{fig_cvg}, we  represent the evolution of $u^{(n)}(t)$ for $t\in [0, 80]$: each curve corresponds to a single run on one sampled graph.  In the left panel, we consider  a  population  of  size $n=n_0$,  with  $n_0=2\,  000$,  and
different values  of $w^{(n)}_I$:
as expected, one can clearly  see  that  the deviation of $u^{(n)}$  from  $u_*$ 
is small, compared to the temporal fluctuations, when $w^{(n)}_I$ is close to $0.1$  and increases to much higher values as $w^{(n)}_I$ grows closer to $2$.
In the  right panel, we  consider $w^{(n)}_I\asymp n^{-\alpha}$ 
for $\alpha=0.3$  and different  values  of  $n$: one  can  clearly see   
that the discrepancy is reduced as  $w^{(n)}_I$ gets closer to~0. 
For $n=10^5$,
the temporal fluctuations are almost not visible on the plot, 
while the discrepancy with $u_*$ is small but still noticeable.

Fig.~\ref{fig_cvg} shows that,  for  large
population size  $n$, the trajectory  of $u^{(n)}$ has small  fluctuations, which
seem  much smaller than  the discrepancy between $u^{(n)}$  and
$u_*$.
For this reason, we choose to investigate the temporal fluctuations of $u^{(n)}$, which can be interpreted  as  an internal variance term, in Section~\ref{sect:internal-noise}
and the discrepancy between a temporal average of $u^{(n)}$
and $u_*$, which can be interpreted as a bias term, in Section~\ref{sect:bias}.

\subsubsection{An internal noise:  the temporal fluctuations of  $u^{(n)}$ for large population size $n$}\label{sect:internal-noise}
\hfill\\

\begin{figure}
	\centering
	\begin{subfigure}[b]{0.48\textwidth}
		\centering
		\includegraphics[width=\textwidth]{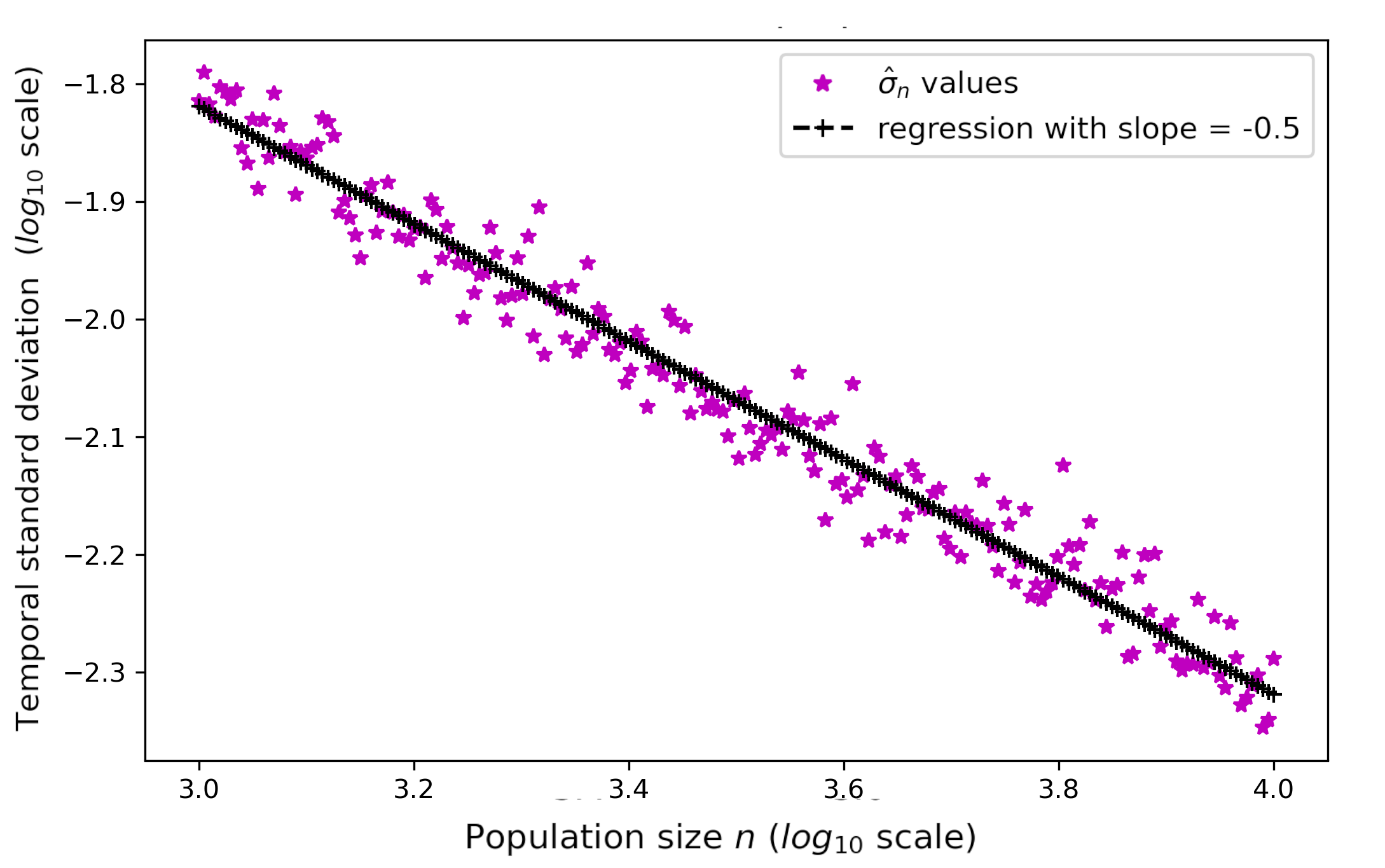}
		\caption{\centering $\alpha = 0.3$.}
	\end{subfigure}
	\hfill
	\vrule
	\hfill
	\begin{subfigure}[b]{0.48\textwidth}
		\centering
		\includegraphics[width=\textwidth]{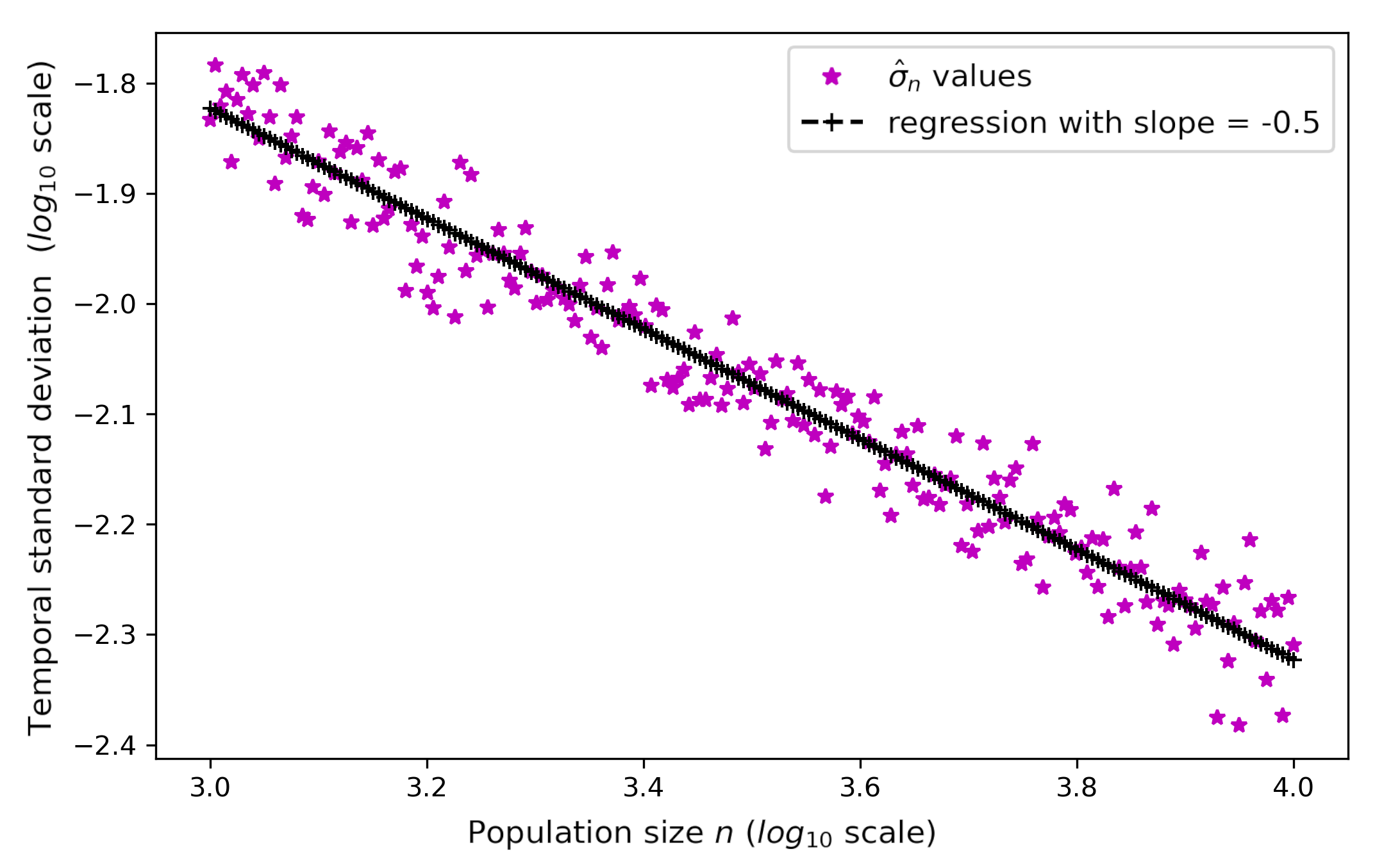}
		\caption{\centering $\alpha =  0$.}
	\end{subfigure}
	\caption{Temporal fluctuations: 
		temporal standard deviation $\hat \sigma_n$ of the proportion of infected individuals (defined in \eqref{eq:tempo-av}), and comparison with the decay $1/\sqrt{n}$,  when $w^{(n)}_I\asymp n^{-\alpha}$.  Each  star point  is obtained from a single run.}
	\label{fig_tpl_std}
\end{figure}

In order to measure the temporal  fluctuations of the  proportion  of
infected  individuals,  we
compute the standard  deviation of the data points  in the time-interval
$I_0=[20,  80]$. 
More precisely, the  random
temporal  average  $\hat  u^{(n)}_*$  and  the corresponding  standard  deviation
$\hat{\sigma}_n$ of $u^{(n)}$ over the time interval $I_0$ of length $|I_0|=60$ are given
by:
\begin{equation}
	\label{eq:tempo-av}
	\hat u^{(n)}_*=\frac{1}{|I_0|} \, \int_{I_0} u^{(n)}(t)\, dt
	\quad\text{and}\quad
	\hat{\sigma}_n^2= \frac{1}{|I_0|} \, \int_{I_0} \Big(u^{(n)}(t)- \hat
	u^{(n)}_*\Big)^2\, dt. 
\end{equation}

This  computation is
motivated by  the Orstein-Uhlenbeck  description of the  fluctuations in
the compartmental SIS  model, corresponding to $w^{(n)}_E=1$,  which are known to be of
order  $1/\sqrt{n}$            (see             \cite[Theorem
2.3.2]{ballbrittonlaredopardouxsirltran}).  
In  Fig.~\ref{fig_tpl_std},  we plot  the temporal  standard deviation
$\hat \sigma_n$ as a function of the population size $n$.
The  resulting relation  is compared  to the  prediction of  the central
limit  theorem  for   mean-field  interactions,  that  is   a  decay  as
$1/\sqrt{n}$.          We  consider  the  sparse  case  $\alpha  =  0.3$  
in  the left  panel, and the  very sparse
case  $\alpha  =0$  (so that  $w^{(n)}_I  =1.2$ for all $n$)  
in  the  right  panel.   The
prediction of a decay as $1/\sqrt{n}$  is clearly confirmed, even in the
very sparse case.

\subsubsection{The bias of the temporal average of  $u^{(n)}$ for different sparsity levels}\label{sect:bias}
\hfill\\
\begin{figure}
	\centering
	\begin{subfigure}[b]{0.48\textwidth}
		\centering
		\includegraphics[width=\textwidth]{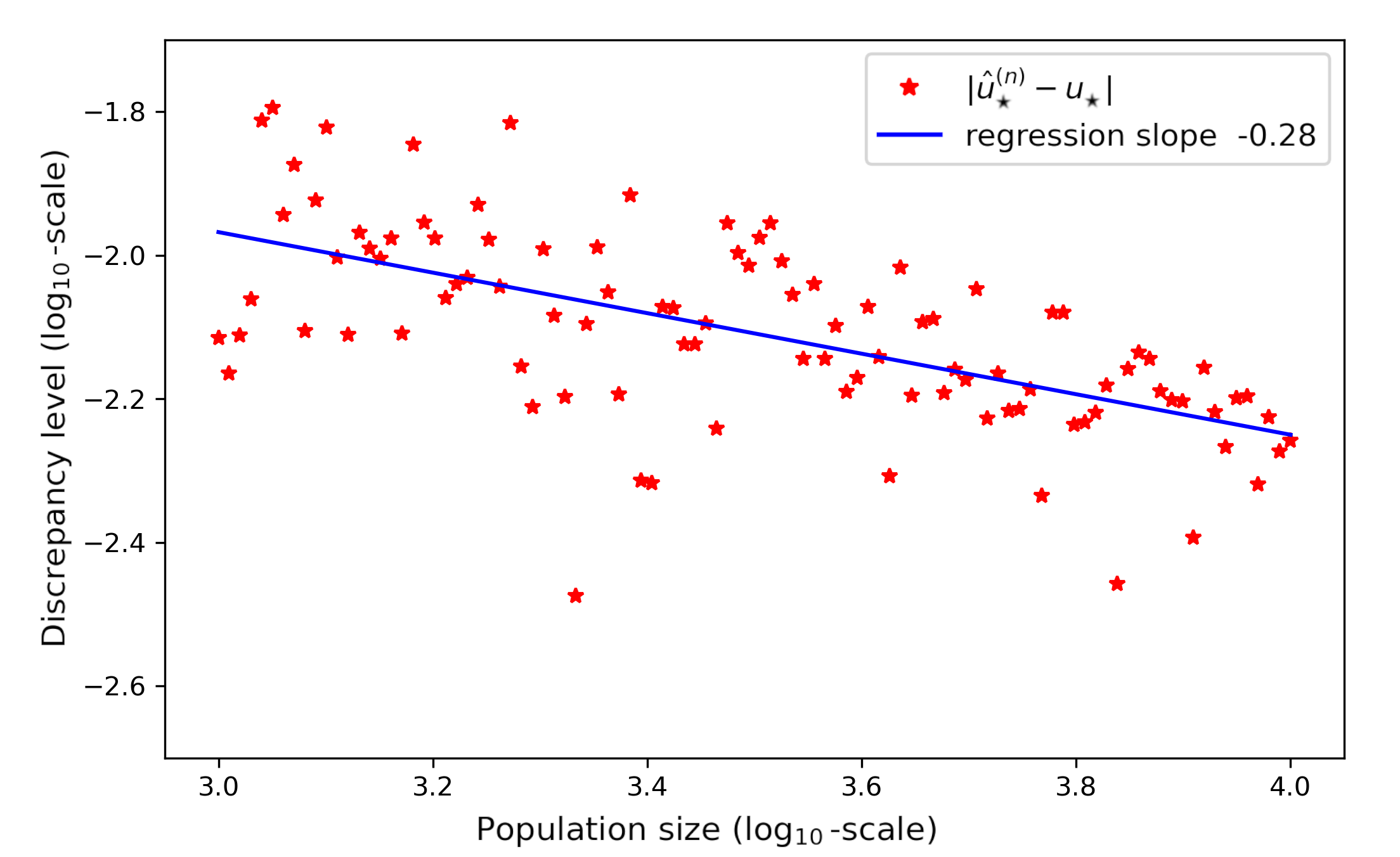}
		\caption{$|\hat{u}^{(n)}_* -  u_*|$ for growing population size $n$ (in log-log scale)
			and regression slope,  for $\alpha = 0.3$.\\
			$R^2\approx 0.38$ is the coefficient of determination corresponding 
			to the proportion of variance 
			captured by the prediction with a slope of $-0.3$
			(and adjusted averages).}
	\end{subfigure}
	\hfill
	\vrule
	\hfill
	\begin{subfigure}[b]{0.48\textwidth}
		\centering
		\includegraphics[width=\textwidth]{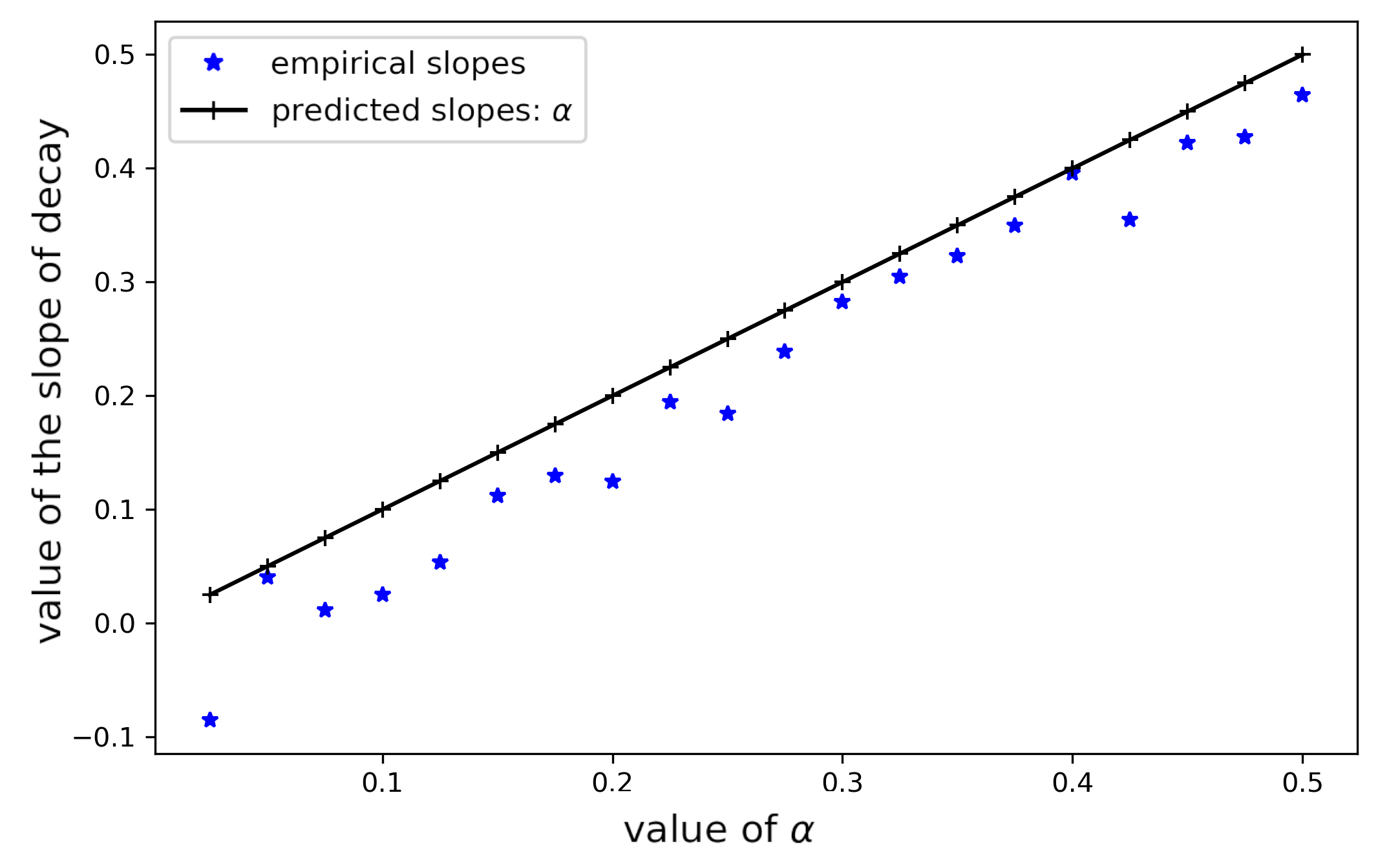}
		\caption{Slopes of the log-log regressions of $|\hat{u}^{(n)}_* -  u_*|$ vs.~$n$, 
			for different values of $\alpha$.\\
			$R^2\approx 0.97$ is the coefficient 
			of determination corresponding 
			to the proportion of variance captured by the prediction of slopes given by exactly $-\alpha$.}
	\end{subfigure}
	\caption{Bias term: $|\hat{u}^{(n)}_* -  u_*|$ and its decay rate for growing population size when $w^{(n)}_I\asymp n^{-\alpha}$,
		and comparison with the decay rate 
		of $\log|\hat{u}^{(n)}_* -  u_*|$ as $-\alpha \log(n)$.}
	\label{fig_alpha_cvg_rate}
\end{figure}

To obtain  more reliable estimations  of the decay in  the discrepancy
between the random process $u^{(n)}$ and the equilibrium $u_*$, we focus in 
Fig.~\ref{fig_alpha_cvg_rate} on the
discrepancy between  the 
random temporal average $\hat  u^{(n)}_*$ 
and $u_*$.
In the  left  panel of Fig.~\ref{fig_alpha_cvg_rate}, we thus
plot the deviation
$|\hat{u}^{(n)}_* -  u_*|$  in a  log-log scale  as a  function of
population size~$n$,  for $\alpha = 0.3$. Each of the $N=100$  points
is obtained from a single run
(with the same relation for $w_I^{(n)}$ as in Fig.~\ref{fig_tpl_std}). 
We then find the regression slope of such points, which turns out to be -0.28, not far from  $-\alpha = -0.3$. 

In the  right panel, we plot these decay slopes for different values of $\alpha$. Each point in the  right panel is an empirical regression slope of $|\hat{u}^{(n)}_* - u_*|$ in $n$ with a log-log scale, obtained with the same procedure used in left panel  for $\alpha=0.3$, requiring $N=100$ runs.
Based on our analysis, we expect that $n^{-\alpha}$ is the maximal level of deviation that can be due to the graph structure, in view of the bias term in the  coupling argument (see Proposition \ref{pr_cpl_G1}); whereas $1/\sqrt{n}$ is the
internal fluctuations level in the martingale term (see Equation~\ref{eq:martingalepb})
that dominates the deviations in the mean-field case (see Fig.~\ref{fig_tpl_std}).
Provided $\alpha <0.5$, the bias term 
$n^{-\alpha}$ is large against the fluctuation term $1/\sqrt{n}$.
This intuition is confirmed in the right panel of Fig.~\ref{fig_alpha_cvg_rate}
where the empirical slopes are very close to a linear growth in $\alpha$, so long as $\alpha\leq 0.5$.
The quality of this approximation of linear growth is also quantified,
as indicated in the legend,
by the coefficient of determination $R^2 \approx 0.97$,
very close to 1.

Regression slopes for values of $\alpha$ 
larger than $0.5$
have also been computed
and their values suggest that this linear curve
should be extended.
This observation is still compatible 
with the conjectured order of variation 
of the bias and of the fluctuations
due to the presence of multiplying factors,
so that the values of $n$ in the simulations are 
presumably too small 
to observe  the predominance of the latter.
For this reason, 
we restrict the right panel to $\alpha\leq 0.5$.

\subsection{The very sparse case}
\label{sec_Vsparse}
\hfill\\
We  consider the  very sparse  case  $\alpha=0$ (and  $a=1$), where  the
number  of edges  in the  graph is  of the same order  as the  size of  the
population:  $n\, w^{(n)}_E\asymp   1$,  and  thus  $w^{(n)}_I\asymp   1$  (since
$n\cdot w^{(n)}_E\cdot  w^{(n)}_I=w=3$).  In  particular Assumption~\ref{hyp:techn} fails.
By means of  simulations, we observe that  the integro-differential equation \eqref{uT} is not a  reliable description of the  dynamics for
$n$ large in this case, thus highlighting the importance of Assumption~\ref{hyp:techn}.

More precisely, we consider constant $w^{(n)}_I$ (and hence constant $n\, w^{(n)}_E = w /w^{(n)}_I$), for values of $w^{(n)}_I$ that will be specified below but remain smaller than $w=3$.
Notice that this falls in the very sparse case, where Assumption~\ref{hyp:techn} fails; the corresponding Erd\H{o}s-R\'enyi graph is below the connectivity threshold and in the super-critical regime $n\, w^{(n)}_E\geq c>1$, where there is a unique giant component containing a fraction of the population, and all other components have a size growing at most as $\log n$.

The intuition suggests that the infection dies out rapidly 
rapidly on all components other than the giant one (due to their small size), 
so that the effective population is rather
given by  the size of the  giant component  and  not the size of the
whole population. For this reason, we also study the proportion of infected
individuals at time $t$ in the giant component, which we will denote by $v^{(n)}(t)$
to avoid confusion with $u^{(n)}(t)$. 
To remove the fluctuations, 
we consider the  temporal averages $\hat u^{(n)}_*$ as
in~\eqref{eq:tempo-av} and $\hat v^{(n)}_*$ with analogous definition, replacing $u^{(n)}$ by $v^{(n)}$. We compare both $\hat u^{(n)}_* $ and $\hat v^{(n)}_*$ with $u_*$.

Our simulation results are shown in Fig.~\ref{fig_sparse}. Each  point  in  Fig.~\ref{fig_sparse}  
is derived from the average in time,
either $\hat u_*^{(n)}$ in purple or $\hat v_*^{(n)}$ in red,
over a single run of epidemic.
Intervals of  fluctuations that account for twice  the temporal standard deviations (either $2 \hat \sigma_n$
or the analogous value estimated 
only over the individuals belonging to the giant component)
are displayed with whiskers of the corresponding color.
These intervals of fluctuations
demonstrate that these temporal fluctuations play a negligible role in the trends.

\begin{figure}
	\centering
	\begin{subfigure}[b]{0.48\textwidth}
		\centering
		\includegraphics[width=\textwidth]{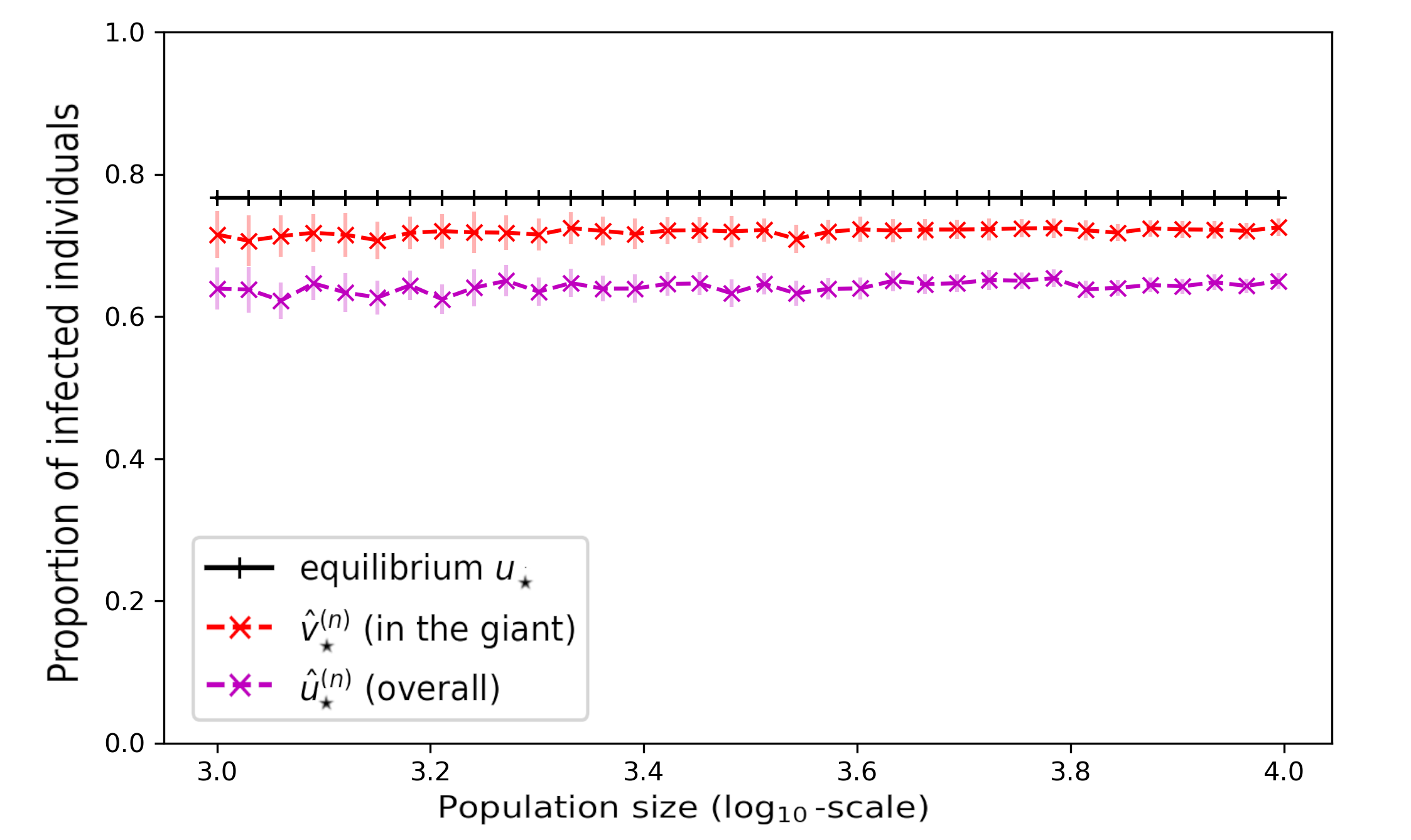}
		\caption{Temporal average of the proportion of infected individuals
			for growing population size $n$ (in log-scale),
			with $w^{(n)}_I =1.2$ and   $n\cdot w^{(n)}_E=2.5$.}
	\end{subfigure}
	\hfill
	\vrule
	\hfill
	\begin{subfigure}[b]{0.48\textwidth}
		\centering
		\includegraphics[width=\textwidth]{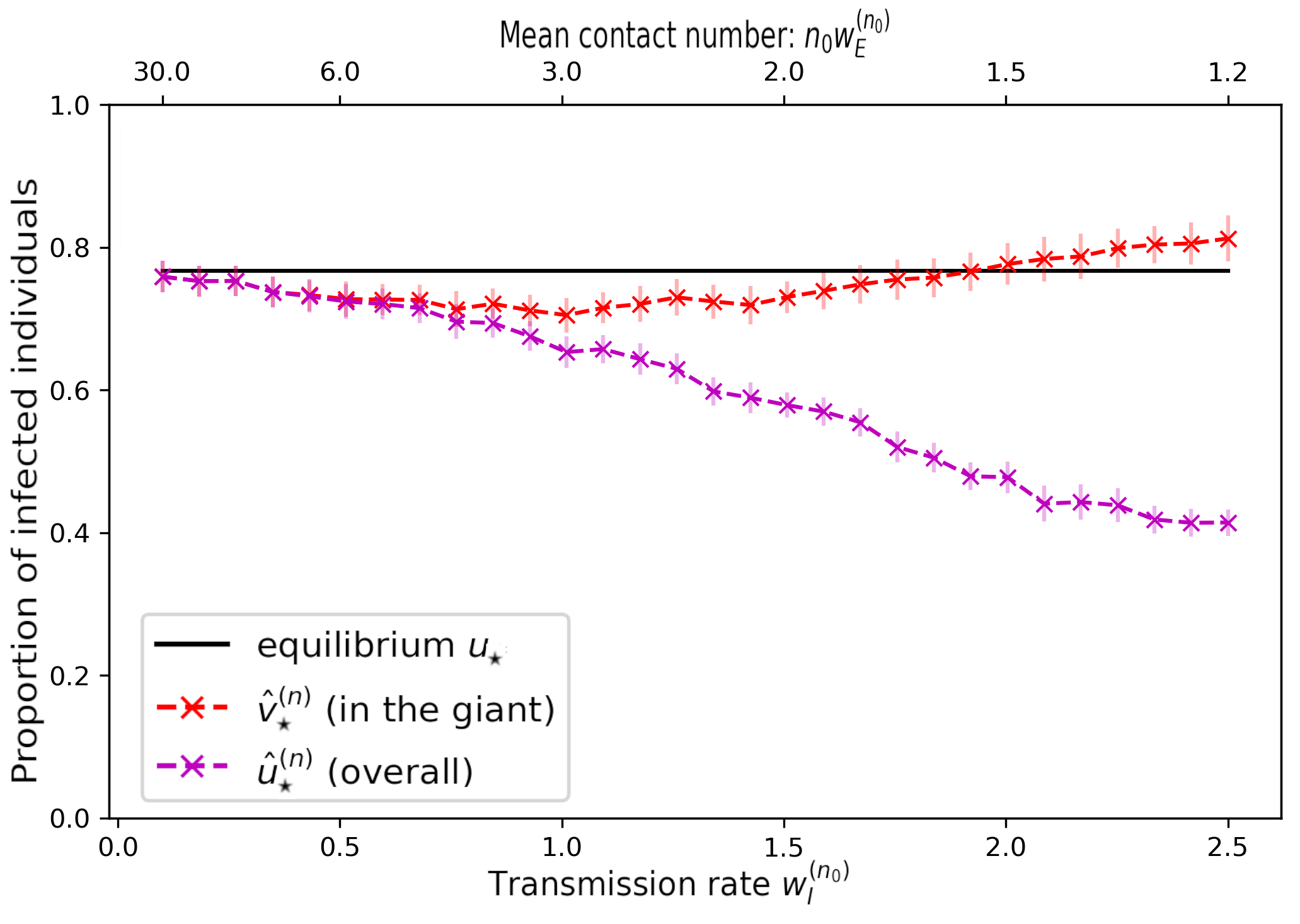}
		\caption{Temporal  average of the proportion of infected individuals  as  a  function  of
			$w^{(n_0)}_I$,  with   fixed  population size
			$n_0$ and $w^{(n_0)}_E \cdot w^{(n_0)}_I=3/n_0$.}
	\end{subfigure}
	\caption{Very sparse yet super-critical regime: $n \cdot w^{(n)}_E = c$, with $c>1$. 
		Temporal averages $\hat  u^{(n)}_*$ of the proportion of infected individuals w.r.t.\ the total population and $\hat  v^{(n)}_*$ w.r.t.\ population of the giant component
		(each with whiskers for the $\pm$2-standard-deviation intervals,  with temporal standard deviation as in \eqref{eq:tempo-av}),  and comparison with the equilibrium $u_*$.}
	\label{fig_sparse}
\end{figure}

In  the left  panel of  Fig.~\ref{fig_sparse}, we  present the  temporal
average  of  the  proportion  of  infected  individuals  for  the  total
population  (that is,  $\hat u^{(n)}_*$)  and the  population in  the giant
component (that is, $\hat  v^{(n)}_* $), as a function of  $n$, in the very
sparse regime  $\alpha=0$ (with $w^{(n)}_I  =1.2$).  Recall that  $w=3$, so
that $n\, w^{(n)}_E=2.5$ corresponds indeed  to the supercritical regime, but
the graph is not connected.  We  observe a convergence of the quantities
$\hat  u^{(n)}_*$ and  $\hat v^{(n)}_*$  as  $n$ goes  to infinity,  but to  a
smaller value  than the equilibrium $u_*$  predicted by  the model
when  the graph  is  dense  or merely sparse. 
The  restriction  to the  giant
component  entails a  smaller  discrepancy,  implying that  this
effect  is   indeed  contributing;   however  the   discrepancy  remains
significant.  This is all the more  noticeable as the restriction to the
giant component  corresponds to  a conditioning  of the  graph structure
towards individuals  having more contacts  between each other.

In the right panel of  Fig.~\ref{fig_sparse}, the size of the population
is fixed to $n=n_0$, and we study the effect of $n\, w^{(n)}_E$ getting close
to the transition phase $c=1$,  where the giant component disappears, by
letting $w^{(n)}_I$  grow near $3$. Similarly to the left panel, we plot the temporal  average of
the proportion  of infected individuals  for the total  population (that
is, $\hat u^{(n)}_*$)  and the population in the giant  component (that is,
$\hat v^{(n)}_* $), but here as a function  of $w^{(n)}_I$, for fixed $n=n_0$.
When $w^{(n)}_I$ is close to 0, the coupling approximation is valid, that  is,
Assumption~\ref{hyp:techn} holds, and almost  all the population is in the giant component:  thus we can apply our result  and as expected the
values of $\hat u^{(n)}_*$, $\hat v^{(n)}_*$ and $u_*$ coincide.
When we start increasing $w^{(n)}_I$, as already observed in the left panel, a discrepancy starts appearing between $\hat u^{(n)}_*$ and $u_*$, and between $\hat v^{(n)}_*$ and $u_*$, the latter being smaller.
When increasing $w^{(n)}_I$ further, the effect of restricting the attention to the giant component becomes stronger: the
discrepancy  between $\hat u^{(n)}_*$ and $u_*$ keeps increasing,
while the one between $\hat v^{(n)}_*$ and $u_*$ remains small, and also changes sign when $w^{(n)}_I$ is larger than 2.
The increase in the discrepancy when considering the whole population is  largely 
due to a larger proportion of individuals outside of the giant
component, among which the epidemic
is not able to sustain itself.

\section{Summary of contributions}
\label{sec_cl}

Starting from a stochastic individual-based model of a finite population
on a general state space  of interactions, we established limit theorems
showing that in large populations, the complex combination of randomness
coming both from  the graph structure and from the infection
events  can be treated as
a  small perturbation  of a  deterministic process.   This deterministic
process    is    described    by   an integro-differential equation,
see~\eqref{eq:intro-uT}  introduced  in  \cite{DDZ22a}. This integro-differential equation  is ruled  by a  bounded recovery  rate $\gamma$  and a
bounded  transmission kernel    $w$  that combines  these  two  aspects:  the
connection density 
and the  infection  rate.

Due
to  possible compensations  between  these  two aspects, we  do not  rely on  any form of  convergence of  the random
graph sequence by itself.  Instead,  we propose an efficient coupling of
the  dynamics on  a  random graph  structure (fixed  in  time) with a
dynamics on a complete graph, in which
there is no persistent graph structure and 
only the infections are random.   By  averaging over  the
proportion  of individuals  whose  state can disagree  between the  two
dynamics, we justify that their limits necessarily  coincide as the
size $n$ of the population tends  to infinity. The crucial condition for
the coupling to provide a good approximation is that the average rate of
infection over the population goes down to
$0$    as    $n$    tends    to    infinity~\eqref{eq:intro-wI=0}.

We rely on very mild assumptions  on the recovery rate function $\gamma$
and the transmission kernel $w$. This  allows to consider for $w$ a wide
variety of  potential interactions  networks (see Remark~\ref{rem_Vgraphs}  and Assumption~\ref{hyp:A1}), such  as discontinuous
geometric kernels,  continuous features  on a multidimensional  setting, and stochastic block models.  
Let us also stress that there is no assumption on the connectivity of the
random graph $G^{(n)}$ of the population of size $n$: connectivity is not a relevant property for our convergence result. 

Our convergence  result holds in the  case of dense and  sparse graphs,
where the total number of edges is,  for a large population of size $n$, of order  $n/\epsilon_n$ with $\lim_{n \rightarrow\infty} \epsilon_n=0$, thus for example 
of order  $n^a$ with  $a=2$ (dense) and $a\in  (1,2)$ (sparse); see Remark
\ref{rem_nAlpha}.  Assumption~\ref{hyp:techn} fails in the case of very
sparse graphs  $a=0$, for which the  number of edges is  of order $O(n)$
and thus the average rate of infection over the population is of order 1.  It  appears that Assumption~\ref{hyp:techn}
is crucial to ensure 
convergence 
as it is shown in simulations in Section~\ref{sec_Vsparse}.

Notice that the model could be easily generalized to allow the inclusion of non-Markovian
description of the    disease    propagation, see
\cite{arazozaclemencontran,     barbourreinert,     forienpangpardoux,
	forienpangpardouxzotsa, foutelrodierblanquartetal} in these directions.  Also, there is no motion nor addition  or deletion of individuals. For epidemics on  moving particles, we
can refer  to \cite{bowongemakouapardoux,durrettyao2022}  for instance. 
Another important direction that could be included would be to relax the consideration of a static (in time)
graph. When the epidemic spreads, the social network can be affected: either by preventive measure performed by separate individuals, who may drop edges and/or rewire connections depending on how the diseases progresses (e.g. \cite{ballbritton22}), or by public health measures such as lockdowns or contact-tracing (e.g. \cite{charpentierelielaurieretran,kryvenstegehuis}). This direction is also very important for future research.
\\

\subsubsection*{Codes}
All the codes and files necessary to reproduce the results presented in this document are available in the GitHub repository:	
\url{https://forgemia.inra.fr/aurelien.velleret/simulations_containment_strategies_and_city_size_heterogeneity}

Simulations and figures have been created in Python (v3.8.10) and many associated simulation outcomes are saved as csv files. Please refer to the Readme file for further information.

\appendix

\section{Appendix: Proof of Proposition~\ref{pr_cpl_G1}}
\label{sec:appendix}

We assume in this section that the hypotheses of
Proposition~\ref{pr_cpl_G1} hold.

\subsection{The epidemic process $\widetilde{\eta}^{(n)}$ coupled to $\eta^{(n)}$}
\hfill\\
\newcommand{\arr}{\overrightarrow{a}}
Recall the notations  of  Section~\ref{sec:proofs}. Following
the classical graphical approach for defining interacting particle systems,
see \cite[Section III.6]{Lig05} for reference, let us define a
graphical model for the coupling as follows. Let $(V_{\ell}(i,
j))_{1\le i<j, \ell \in \mathbb{N}^*}$  be a  family of  independent uniform
random  variables  on   $[0,  1]$,  and  independent   of  $\mathcal{X}^{(n)}$  and
$\mathcal{E}^{(n)}$. We let $V_1(i,j)$ be equal to $V(i,j)$ from
Section~\ref{section:IBM}. Similarly to Section~\ref{section:IBM},
we set $V_\ell (j,i)=V_\ell (i,j)$  and $V_\ell (i,i)=0$
for  convenience.

For each atom $(s,i,j,u)$ of the Poisson measure $Q_I$, we draw an
\emph{arrow} from $j$ to $i$ at time $s$  if $u\leq  w^{(n)}_I(i,j)$,
and denote the corresponding event by  $\{u\leq  w^{(n)}_I(i,j)\}$. 
Let us denote by $N_t^{(n)}(i,j)$ the number of arrows from $j$ to $i$
or vice-versa
up to time~$t$: 
\begin{equation}
	\label{Ntij}
	\mathrm{N}_t^{(n)}(i_0, j_0)= \int \mathbbm{1}_{\{s\leq t\}}\, \mathbbm{1}_{ \{
		\{i,j\}=\{i_0,j_0\}\}} \mathbbm{1}_{\{u\leq w^{(n)}_I(i,j)\}}\,   dQ_I.
\end{equation}
For any atom of $Q_I$ that leads to an arrow, we define the connection events:
\begin{align}
	\label{eq:def-AC}
	C^{(n)}(i,j,u,s)
	&=  \{u\leq  w^{(n)}_I(i,j)\} \cap \{ V_1(i,j) \leq
	w^{(n)}_E(i,j)\} =  \{u\leq  w^{(n)}_I(i,j)\} \cap \{i\sim j\},  \\
	\label{eq:def-tC}
	\widetilde{C}^{(n)}(i,j,u,s)
	&= \{u\leq  w^{(n)}_I(i,j)\} \cap  \{ V_{N^{(n)}_s(i,j)}(i,j)  \leq w^{(n)}_E(i,j) \}.
\end{align}

On $C$ (resp. $\widetilde{C}$), we say that the arrow is activated for the
original process (resp. for the coupled  process). 
As usual for graphical models, $\eta^{(n)}$ (resp. $\widetilde{\eta}^{(n)}=( \widetilde{\eta}^{(n)}_t)_{t\in
	\mathbb{R}_+}$) is then 
constructed by following the activated arrows defined by 
$C^{(n)}$ (resp. $\widetilde{C}^{(n)}$). 
Compared to \eqref{eq:def-B},  the dynamics of $\widetilde{\eta}^{(n)}$ is the same as
the one of $\eta^{(n)}$  except that the
event $\{i\sim j\}\cap \{ u\le w^{(n)}_I (x_i, x_j) \} $ is replaced by the
connection event  $\widetilde{C}^{(n)}(i,j,u,s)$.  If the arrow is the first
to occur between $i$ and $j$, then $N^{(n)}_s(i,j)=1$ and the events are the same,
if not, the use of an independent $V_\ell(i,j)$ in the definition of
$\widetilde{C}^{(n)}$ corresponds to a resampling of the connection between $i$ and $j$.

\medskip

Thus, the    epidemic   process
$\widetilde{\eta}^{(n)}=( \widetilde{\eta}^{(n)}_t)_{t\in
	\mathbb{R}_+}\in   \mathcal{D}$   on  the   $n$   individuals
$\mathcal{X}^{(n)}=(x_i)_{i\in [\![1, n]\!]}$ is defined by:
\begin{equation}
	\label{eq:te-def}
	\widetilde{\eta}_t^{(n)}(dx, de) = \frac{1}{n} \sum_{i =1}^{n}
	\delta_{(x_i,\widetilde{E}^i_t)}(dx, de) ,
\end{equation}
with $\widetilde{\eta}^{(n)}_0=\eta^{(n)}_0$ and, similarly to~\eqref{SDE:etaN}, for $t>0$:
\begin{multline}
	\label{SDE:etaNtilde}
	\widetilde{\eta}^{(n)}_t - \widetilde{\eta}^{(n)}_0
	= n^{-1} \int \mathbbm{1}_{\{s<t\}} \, (\delta_{(x_i, S)} - \delta_{(x_i, I)})
	\,\mathbbm{1}_{\widetilde{A}^{(n)}(i,u,s)} \,  dQ_R\\
	+ n^{-1} \int \mathbbm{1}_{\{s<t\}}\,   (\delta_{(x_i, I)} - \delta_{(x_i,
		S)}) \, \mathbbm{1}_{\widetilde{B}^{(n)}(i,j,u,s)} \,  dQ_I,
\end{multline}
where $\widetilde{A}^{(n)}(i,u,s)$ and $\widetilde{B}^{(n)}(i,j,u,s)$ are defined
similarly to~\eqref{eq:def-A} 
and~\eqref{eq:def-B}:

\[
\widetilde{A}^{(n)}(i,u,s)=
\{i\leq  n\}\cap \{u\le  \gamma^{(n)}(x_i)\}\cap\{\widetilde{E}^i_{s-}  = I\},
\]
and
\[
\widetilde{B}^{(n)}(i,j,u,s)=  \{i,j\leq  n\}
\cap \widetilde{C}^{(n)}(i,j,u,s)
\cap
\{\widetilde{E}^i_{s-} = S, \widetilde{E}^j_{s-} = I\}.
\]

From this description, we get that $\widetilde{\eta}^{(n)}$ is an epidemic
process on a complete graph (so the corresponding
connection density 
is $\widetilde{w}^{(n)}_E\equiv 1$)  with the infection rate 
$\widetilde{w}^{(n)}_I=w^{(n)}_I  w^{(n)}_E = n^{-1}\, w^{(n)}$ and recovery rate
$\widetilde{\gamma}^{(n)}=\gamma^{(n)}$.

\subsection{The fog process}
\hfill\\
To  study the  coupling  between  $\eta^{(n)}$ and  $\widetilde{\eta}^{(n)}$,
we introduce a process that we call the \emph{fog} process, that we first
describe informally. 
At time  $t$, each vertex $i$
is either  in or out of the fog: $\xi^{(n)}_t(i) = 1$ or~$0$. At the beginning
there is no fog ($\xi^{(n)}_0(i) = 0$). Once a vertex enters the fog it stays
foggy forever. The rules for creating the fog or propagating it will ensure
the following key property:
\[ \text{If $\xi^{(n)}_t(i) = 0$, then the  processes
	$\eta^{(n)}$ and $\widetilde{\eta}^{(n)}$ coincide at vertex $i$
	on the time interval $[0,t]$}.\]
In other words, the fog is an upper bound on the vertices where the processes
$\eta^{(n)}$ and $\widetilde{\eta}^{(n)}$ may have decoupled. 
The fog process is formally  a family of 
pure-jump  processes $\xi^{(n)}(i)=(\xi^{(n)}_t(i))_{t\in \mathbb{R}_+}$, that start
at $0$ and jump at most once, to the value $1$.
Consider the following conditions:
\begin{enumerate}[label=\alph*)]
	\item The vertex $i$ is not currently in the fog: $\xi_{t-}^{(n)}(i) = 0$.
	\item There is an arrow at time $t$ from some $j$ to $i$.
	\item\label{it:child}
	The vertex $j$ is in the fog, and the arrow is activated for the
	coupled  process. 
	\item[c')]\label{it:root}
	The  arrow is activated for exactly one of the two processes.
\end{enumerate}
If a), b) and c) are satisfied, we say that the fog  \emph{propagates} from $j$ to $i$, and that $i$ is the \emph{child} of $j$; otherwise, 
if a), b) and c') hold, we say
that $i$ is a \emph{root} for the fog process. In both cases $i$ enters the fog
at time $t$~: $\xi_t ^{(n)}(i) = 1$.
Formally the propagation of the fog, that is conditions a), b) and c), corresponds to the event:
\begin{equation}
	\label{eq:Hprop}
	H^{(n)}_\mathrm{prop}(i, j, u,  s) = \{i, j \le n\}\cap  \{\xi^{(n)}_{s-}(i) = 0\} 
	\cap \widetilde{C}^{(n)}(i,j,u,s) \cap \{ \xi^{(n)}_{s-}(j) = 1 \},
\end{equation}
and the conditions  a), b) and c') to the event:
\begin{equation*}
	H_\mathrm{xor}^{(n)}(i, j, u,  s) =\{i, j \le n\}\cap    \{\xi^{(n)}_{s-}(i) = 0\} 
	\cap  \Big( 
	C^{(n)}(i,j,u,s) \triangle\widetilde{C}^{(n)}(i,j,u,s)\Big), 
\end{equation*}
where
$A \triangle  B= (A \cap B^  c) \cup (A^c \cap  B)$. In particular
becoming a root corresponds to the event:
\begin{equation}
	\label{eq:def-Hor}
	H^{(n)}_\mathrm{root}(i, j, u,  s) =  H^{(n)}_\mathrm{prop}(i, j, u,  s)^c
	\cap H_\mathrm{xor}^{(n)}(i, j, u,  s). 
\end{equation}
In particular, notice that if at time  $s$, $j$ is in the fog, the arrow
from $j$ to $i$ is activated and $i$ was not previously in the fog (that
is we are on the  event $ C^{(n)}(i,j,u,s) \cup \widetilde{C}^{(n)}(i,j,u,s)$), then
$i$ is now in the fog (this corresponds to the events a), b) and c) or c')).
\medskip

We can formalize the evolution of $\xi^{(n)}$ as follow, for $i\in [\![1, n]\!]$:
\begin{equation}
	\label{xi2n}
	\xi^{(n)}_t(i')
	= \int \mathbbm{1}_{\{s\leq t\}}\, 
	\, \mathbbm{1}_{H^{(n)}(i, j, u,  s)}  \, \mathbbm{1}_{\{i=i'\}}\,
	dQ^{(n)}_I,
\end{equation}
where $H^{(n)}(i, j, u,  s)$ specifies that 
$i$ enters the fog at time $s$: 
\[
H^{(n)}(i, j, u,  s) =   H^{(n)}_\mathrm{prop}(i, j, u,  s)  \cup   H^{(n)}_\mathrm{root}(i, j, u,  s) .
\]

Let us denote the number of vertices in the fog at time $t$ by:
\begin{equation}
	\label{eq:def-Xi}
	\Xi^{(n)}_t=\sum_{i\in [\![1, n]\!]}\xi^{(n)}_t(i).
\end{equation}

The fog process is designed to ensure the following upper bound. 
\begin{lem}[Control of the coupling by the fog process]
	\label{lem:xieND}
	The following upper-bound holds a.s. for all $T\geq 0$:
	\[
	\textstyle{\sup_{\{t\le T\} }  }\, \|\eta^{(n)}_t - \widetilde{\eta}^{(n)}_t\|_{TV} \leq  \frac{1}{n}  \, \Xi^{(n)}_T.
	\]
\end{lem}

\begin{proof}
	Since by definition $\eta^{(n)}_0=\widetilde{\eta}^{(n)}_0$ and $\xi^{(n)}_0=0$ and that the
	same individuals in  the two epidemic processes  when infected recover
	at the same time (the recovery procedure is driven by the same Poisson
	point measure $Q_R$), we get by construction that the number of points
	in the support of $\xi^{(n)}_t$  is then, after careful consideration, an
	upper  bound  of the  number  of  vertices  with different  states  in
	$\eta^{(n)}_t$  and $\widetilde{\eta}^{(n)}_t$.   More  precisely,  the following upper-bound holds a.s.\ for  all
	$t\geq 0$ and all $i\in [\![1, n]\!]$:
	\[
	\mathbbm{1}_{\{E^i_t \neq \widetilde{E}^i_t\}} \leq  \xi^{(n)}_t(i)\in \{0, 1\}.
	\]
	Using that  the process  $(\xi^{(n)}_t(i))_{t\in \mathbb{R}_+}$ is
	non-decreasing, and the expressions~\eqref{muteN} for $\eta^{(n)}_t$
	and~\eqref{eq:te-def} for $\widetilde{\eta}^{(n)}_t$, we deduce that:
	\[
	\sup_{t\in [0, T]} \|\eta^{(n)}_t - \widetilde{\eta}^{(n)}_t\|_{TV} \leq  \frac{1}{n} \, \textstyle{\sup_{\{t\le T\} }  }\,
	\Xi^{(n)}_t= \frac{1}{n}  \, 
	\Xi^{(n)}_T.
	\]
\end{proof}

\subsection{Upper bound on the expected size of the fog}
\hfill\\
We follow ideas from 
\cite[Section~1.2]{ballbrittonlaredopardouxsirltran}
where  ``ghost
infections'' are introduced to  associate the SIR infection process
with  a branching  infection process  (the first  being included  in the
second   for    the   chosen   coupling). 
Using the construction of $(\xi^{(n)}_t)_{t\in \mathbb{R}_+}$, we can represent its
support in terms of random forests.
More precisely, recall that each time a vertex enters the fog, it is either
as a \emph{root} or as a \emph{child} of another vertex that is already in the fog
at that time.

We denote by $S_{i}$ the jumping time of $\xi^{(n)}(i)$ (with the
convention that $S_i=+\infty $ if $\xi_\infty ^{(n)}(i)=0$), by
$\mathcal{R}^{(n)}$ the set of roots and  by $\mathcal{R}^{(n)}_t\subset [\![1, n]\!]$ the
set of roots  born up to time $t$:
\begin{equation}
	\label{eq:def-Rnt}
	\mathcal{R}^{(n)}_t=\{i\in \mathcal{R}^{(n)}\, \colon\, S_i \leq t\}.
\end{equation}

For a given root~$i$, let $\xi^{i, (n)}$ denote the process of its
descendants, for $t\geq 0$ and $i'\in [\![1, n]\!]$: 
\[
\xi^{i, (n)}_t(i') = \mathbbm{1}_{\{\text{$i$ is a root}\}} \,
\mathbbm{1}_{\{\text{$i'$ is a descendant of $i$}\}}\, \xi^{(n)}_t(i') .
\]
Let  $\Xi^{i, (n)}_t =  \sum_{i'\in [\![1, n]\!]  } \xi^{i, (n)}_t(i')$  be the  total
population up  to time $t$ fathered  by the root $i$, and notice that
$\Xi^{i, (n)}_t=0$ on $\{t<S_i\}$.  The total size of the fog is therefore:
\begin{equation}
	\label{eq:sum-Xi-j}
	\Xi^{(n)}_t=\sum_{i=1}^n \Xi^{i, (n)}_t\, \mathbbm{1}_{\{i\in \mathcal{R}^{(n)}_t\}} .
\end{equation}

We now give an upper bound on $\mathbb{E}\Big[\Xi^{i, (n)}_t\, \mathbbm{1}_{\{i\in \mathcal{R}^{(n)}_t\}} \Big]$. 
Recall $C_w=\sup_{n\in \mathbb{N}^*} \|w^{(n)}\|_\infty$.

\begin{lem}
	\label{lem_crL}
	The following upper-bound holds for all $t\geq 0$ and $i\in [\![1, n]\!]$:
	\[
	\mathbb{E}\Big[\Xi^{i, (n)}_t \,\mathbbm{1}_{\{i  \in\mathcal{R}^{(n)}_t\}}\Big] \le \mathrm{e}^{C_w\,  t}
	\, \mathbb{P}\Big(i\in \mathcal{R}^{(n)}_t\Big).
	\]
\end{lem}	
\begin{proof}
	To avoid any confusion with the notation $dQ_I=Q_I(ds, di, dj,
	du)$, we shall prove the lemma with $i$ replaced by  $i_0$.
	By  construction
	(recalling \eqref{xi2n}, \eqref{def_RNT}, \eqref{eq:sum-Xi-j} and \eqref{eq_Dni}),
	the evolution of $\xi^{i_0, (n)}$ is given by the following formula, for $i'\in [\![1, n]\!]$:
	\begin{equation}
		\label{eq_Dni}
		\xi^{i_0, (n)}_t(i')
		= 	\mathbbm{1}_{\{i_0\in\mathcal{R}^{(n)}_t\}}
		\mathbbm{1}_{\{i'=i_0\}}
		+ \mathbbm{1}_{\{t\geq S_{i_0}\}}  \int \mathbbm{1}_{\{s\leq t\}}\,
		\mathbbm{1}_{D^{i_0, (n)}( i, j, u,s)} 
		\mathbbm{1}_{\{i'=i\}}\,  dQ^{(n)}_I,
	\end{equation}
	where $D^{i_0, (n)}(i, j, u,s)$ specifies that,
	at time $s$,
	$i$ enters the fog as a child  of a vertex $j$ which, itself, descends
	from  $i_0$ (compare with $  H^{(n)}_\mathrm{prop}$ from~\eqref{eq:Hprop}):
	\[
	D^{i_0, (n)}(i, j, u,s) 
	= \{i, j\le n\}
	\cap \{\xi^{i_0, (n)}_{s-}(i) = 0\}
	\cap \{\xi^{i_0, (n)}_{s-}(j) = 1\}
	\cap \widetilde{C}^{(n)}(i,j,u,s).
	\]

	On the event $\{i_0\in \mathcal{R}^{(n)}\}$, 
	since $n\cdot w^{(n)}_E(x, y)\cdot w^{(n)}_I(x, y)\le C_w$ holds for any $x, y\in \mathbb{X}$,
	the process
	$(\xi^{i_0, (n)}_{S_{i_0}+t})_{t\in \mathbb{R}_+}$ 
	jumps with an additional Dirac Mass at location $i'$
	with a rate upper-bounded at time $t$ by 
	\begin{equation*}
		\sum_{j\in [\![1, n]\!]} \mathbbm{1}_{\lbrace \xi^{i_0, (n)}_{S_{i_0}+t-}(j)\ge 1\rbrace} 
		w^{(n)}_E(x_{i'}, x_j)\cdot w^{(n)}_I(x_{i'}, x_j)\le 
		\frac{C_w}{n}\sum_{j\in [\![1, n]\!]} \mathbbm{1}_{\lbrace \xi^{i_0, (n)}_{S_{i_0}+t-}(j)\ge 1\rbrace}.
	\end{equation*}
	Thus, this process is stochastically dominated by the
	pure-jump process $\zeta^{(n)}=(\zeta^{(n)}_{t})_{t\in \mathbb{R}_+}$ on $[\![1, n]\!]$
	defined by:
	\[
	\zeta^{(n)}_t(i')
	= \mathbbm{1}_{\{i'=i_0\}}
	+  \int \mathbbm{1}_{\{s\leq t\}}\,  \mathbbm{1}_{G^{(n)}(i, j, k, z, s)}
	\mathbbm{1}_{\{i'=i\}}\, 
	Q(ds, di, dj, dk, dz),
	\]
	where $Q$ is a Poisson point measure on $\mathbb{R}_+\times \mathbb{N}^*\times \mathbb{N}^*\times\mathbb{N}^*\times
	\mathbb{R}_+$ with intensity $ds \, \mathrm{n}(di)\, \mathrm{n}(dj)\,\mathrm{n}(dk)\, dz$
	and $G^{(n)}(i, j, k, z, s)$ specifies that, at time $s$,
	an individual is added at vertex  $i$
	as a descendant of an already added individual at vertex  $j$
	(the parameter $k$ being introduced to cope with the allowed multiplicity of individuals at position $j$):
	\[
	G^{(n)}(i, j, k, z, s)
	= \{i, j \le n\}
	\cap \{k \le \zeta^{(n)}_{s-}(j)\}
	\cap \{z\leq C_w/n\}.
	\]

	Let $Z^{(n)}=(Z^{(n)}_t)_{t\in \mathbb{R}_+}$ be the size of the population at time
	$t$ defined by $Z^{(n)}_t=\sum_{i\in [\![1, n]\!]} \zeta^{(n)}_t(i)$. 
	$Z^{(n)}$ is expressed as follows for any $t\ge 0$:
	\[
	Z^{(n)}_t
	= 1+ \int \mathbbm{1}_{\{s\leq t\}}\, \mathbbm{1}_{G^{(n)}(i, j, k, z, s)}\,  dQ
	= 1+ \frac{C_w}{n} \int \mathbbm{1}_{\{s\leq t\}}\, n Z^{(n)}_{s} \, ds + W^{(n)}_t,
	\]
	where $W^{(n)}=(W^{(n)}_t)_{t\in \mathbb{R}_+}$ is a square integrable martingale
	with  quadratic variation:
	\[
	\langle W^{(n)}\rangle_t = \frac{C_w}{n}\int_0^t  n  Z^{(n)}_{s}\,  ds
	=   C_w \int_0^t   Z^{(n)}_{s}\,ds.
	\]
	We recognize the semi-martingale decomposition of a birth process with
	birth rate  $C_w$, started  at $1$. Thus, we have $\mathbb{E}[
	Z^{(n)}_t]=\exp (C_w\, t)$.
	\medskip
	
	We deduce that on the  event $\{i_0\in \mathcal{R}^{(n)}\}$, the total mass
	$\Xi^{i_0, (n)}_{S_{i_0}+t}$ is 
	stochastically dominated  by $Z^{(n)}_t$.  Since the  process $\Xi^{i_0, (n)}$ is
	non-decreasing and 0 on $\{S_{i_0}>t\}$, we deduce that:
	\[
	\mathbb{E}\Big[ \Xi^{i_0, (n)}_{t}\, \mathbbm{1}_{\{i_0\in \mathcal{R}^{(n)}_t\}} \Big]     \leq 
	\mathbb{E}\Big[ \Xi^{i_0, (n)}_{S_{i_0}+t}\, \mathbbm{1}_{\{i_0\in \mathcal{R}^{(n)}_t\}} \Big]     \leq 
	\mathbb{P}\Big(  i_0\in \mathcal{R}^{(n)}_t \Big) \mathbb{E} [Z^{(n)}_t].
	\] 
	This concludes the proof. 
\end{proof}

The following  lemma gives a bound on the probability for a
given vertex to be a root before a given time $t$. 
\begin{lem}
	\label{lem_ERL}
	The following upper-bound holds for all $t\geq 0$ and all $i\in [\![1, n]\!]$:
	\[
	\mathbb{P}\Big(i\in \mathcal{R}^{(n)}_t\, \big|\, \mathcal{X}^{(n)} \Big)  \leq
	\frac{2t(t\vee 1) \, C_w}{n} \, \sum_{j\in [\![1, n]\!],\,  j\neq i}
	\Big\{\big(  w^{(n)}_I(x_i, x_j)  \wedge 1\big) + \big(  w^{(n)}_I(x_j, x_i)
	\wedge 1\big)\Big\}. 
	\]
\end{lem}

\begin{proof}
	Let $t>0$, $n\geq 2$ and $i\in [\![1, n]\!]$ be fixed. By construction, the
	number of activations  $(\mathrm{N}_t^{(n)}(i, j))_{j\in \mathbb{N}^*, i\neq j}$ are independent
	Poisson random variable with respective parameter $t\, W^{(n)}_I(x_i,
	x_j)$ where:
	\begin{equation}\label{def_WIN}
		W^{(n)}_I(x, y)=w^{(n)}_I(x, y) + w^{(n)}_I(y, x).
	\end{equation}
	
	If $i\in \mathcal{R}^{(n)}$ is a root born at (finite)  time $S_i$, we denote by
	$(S_i, i, J_i, U_i)$ the  unique atom  of  the  random measure  $Q_I$  at  time $S_i$,
	and  let $C_i = C(i,J_i,U_i,S_i)$ (resp. $\widetilde{C}_i = \widetilde{C}(i,J_i,U_i,S_i)$)
	be the event that the arrow that leads to the creation of the root $i$
	is activated for the original process (resp. for the coupled  process).
	By the definition of roots, see~\eqref{eq:def-Hor}, 
	the following inclusion holds  for all $i\in [\![1, n]\!]$ and all $t\in \mathbb{R}_+$:
	\begin{equation}
		\label{def_RNT}
		\{i\in \mathcal{R}^{(n)}_t\}
		\subset \{ S_i\leq t\}
		\cap (C_i\triangle \widetilde{C}_i).
	\end{equation}
	By  construction,  see~\eqref{eq:def-AC}  and~\eqref{eq:def-tC},  on  the
	symmetric  difference $C_i\triangle  \widetilde{C}_i$,  the  number of  arrows
	between $\{i, j\}$ up to time $S_i$, that is $L_{i}:=\mathrm{N}^{(n)}_{S_i}(i,J_i)$ must
	be larger than $2$. Thus, the following sequence of inclusions holds:
	\begin{align*}
		\{i\in \mathcal{R}^{(n)}_t\}
		&\subset \{ S_i\leq t\}
		\cap \{ L_{i} \geq 2\} \cap (C_i\cup \widetilde{C}_i) \\
		&\subset \{ S_i \leq t\}
		\cap \{ L_{i} \geq 2\} 
		\cap\Big(\{V_{L_{i}}(i,J_i) \leq
		w^{(n)}_E(x_i, x_{J_i})\} \cup \{V_{1}(i,J_i)\leq
		w^{(n)}_E(x_i, x_{J_i})\}
		\Big)\\
		&\subset\{ N_t^{(n)}(i,J_i)\geq 2\}
		\cap\{\exists \ell\in [\![1, N_t^{(n)}(i,J_i)]\!] \quad \text{such that}\quad  V_{\ell}(i,J_i) \leq
		w^{(n)}_E(x_i, x_{J_i})\} .
	\end{align*}
	
	By construction, for $j\neq i$, the random variables $\mathrm{N}^{(n)}_{t}(i, j)
	$ and  $V_\ell(i,j)$ for $\ell\in \mathbb{N}^*$  are independent and
	the latter are uniformly distributed on $[0, 1]$. We deduce that:
	\begin{multline*}
		\mathbb{P}\Big(i\in \mathcal{R}^{(n)}_ t
		\, \big|\, \mathcal{X}^{(n)}  \Big)\\
		\begin{aligned}
			&\leq  \sum_{j\in [\![1, n]\!],\,  j\neq i} \!\!
			\mathbb{P}\Big(\mathrm{N}^{(n)}_{t}(i, j)\geq 2,\,  \exists \ell\in [\![1, N_t^{(n)}(i,j)]\!]
			\quad \text{such that}\quad  V_{\ell}(i,j) \leq
			w^{(n)}_E(x_i, x_{j})  \, \big|\, \mathcal{X}^{(n)} \Big)\\
			&=\sum_{j\in [\![1, n]\!],\, j\neq i} g\Big(t W^{(n)}_I(x_i, x_j), w^{(n)}_E(x_i, x_j)\Big),
		\end{aligned}
	\end{multline*}
	with:
	\[
	g(\theta,r)= \mathbb{P}_{\theta}\Big(L\geq 2, \,  \exists \ell \in [\![1, L]\!]
	\quad\text{such that}\quad
	V_\ell \le r\Big),
	\]
	where under $\mathbb{P}_{\theta}$, 
	$L$ is distributed as a Poisson random variable with parameter $\theta$ and
	$(V_\ell)_{\ell\in \mathbb{N}^*}$ as independent random variables uniformly
	distributed on $[0, 1]$ and independent of $L$. Elementary computations
	give:
	\begin{align*}
		g(\theta,r)
		= \mathbb{P}_{\theta} \Big(L\geq 2,  \ \exists \ell \in [\![1, L]\!]
		\quad\text{such that}\quad
		V_\ell \le r\Big)
		&  =\sum_{k=2}^\infty \frac{\theta^k}{k!} \mathrm{e}^{-\theta}
		\Big( 1- (1-r) ^{k}\Big)\\
		&= 1- \mathrm{e}^{-\theta r} - \theta r\,  \mathrm{e}^{-\theta}\\
		&\leq  \theta r \, ( \theta \wedge 1),    
	\end{align*}
	where for  the inequality, we used  that $1-\mathrm{e}^{-x}$ is less  than $x$
	and $x  \wedge 1$.  Since   $w^{(n)}=n\, w^{(n)}_I w^{(n)}_E$ is  bounded by
	$C_w$, we get that  $t W^{(n)}_I w^{(n)}_E\leq 2\, t\, C_w/n$
	(recall also the link between $W^{(n)}_I$ and $w^{(n)}_I$ in \eqref{def_WIN} and the symmetry of $w^{(n)}_E$). 
	We obtain  that:
	\[
	\mathbb{P}\Big(i\in \mathcal{R}^{(n)} 
	\, \big|\,\mathcal{X}^{(n)}  \Big)
	\leq \frac{2t\, C_w}{n}  \sum_{j\in [\![1, n]\!],\,  j\neq i}
	\Big( (t w^{(n)}_I(x_i, x_j) ) \wedge 1\Big) + \Big( (t w^{(n)}_I(x_j, x_i)
	) \wedge 1\Big). 
	\]
	This concludes the proof. \end{proof}

\subsection{Conclusion}
\hfill\\
Using Lemma~\ref{lem:xieND}, Equation~\eqref{eq:sum-Xi-j},
Lemmas~\ref{lem_crL} and~\ref{lem_ERL}, as well as the definition of
$\mathcal{I}_n$ in~\eqref{def_HI}, we deduce that:
\[  \mathbb{E}\Big[ \textstyle{\sup_{\{t\le T\} }  }\, \|\eta^{(n)}_t - \widetilde{\eta}^{(n)}_t\|_{TV}\Big]
\leq \frac{1}{n} \mathbb{E}\Big[  \Xi^{(n)}_T\Big]
=\frac{1}{n} \sum_{i=1}^n \mathbb{E}\Big[  \Xi^{i, (n)}_T\Big]
\leq  C_T \, \mathcal{I}_n(w^{(n)}_I\wedge 1),
\]
with $C_T=4  T \,(T \vee 1) \, C_w\, \mathrm{e}^{C_w\,  T}$.
This concludes the proof of Proposition~\ref{pr_cpl_G1}.

\section*{Acknowledgements}
The author would like to thank the anonymous referee for their diligent and professional work. 
\bibliographystyle{alpha}

\bibliography{biblio_vA}

\end{document}